\documentclass[hidelinks,onefignum,onetabnum]{siamart251216}

\ifpdf
\hypersetup{
  pdftitle={Two-level convergence of Algebraic Multigrid with Overlapping Smoothers and Spectral Coarse Grids},
  pdfauthor={O. A. Krzysik, B. S. Southworth, and H. Al Daas}
}
\fi

\usepackage{amsmath,amssymb,amsfonts}
\usepackage{mathtools}
\mathtoolsset{centercolon}

\usepackage{bm}
\usepackage{xcolor}
\usepackage{graphicx}
\usepackage{epstopdf}
\usepackage{algorithmic}
\usepackage{tikz}
\usepackage{array}
\usepackage{stmaryrd}
\usepackage{comment}
\usepackage{accents}

\ifpdf
  \DeclareGraphicsExtensions{.pdf,.png,.jpg,.eps}
\else
  \DeclareGraphicsExtensions{.eps}
\fi

\headers{Two-level LS-AMG-DD convergence}
{O. A. Krzysik, B. S. Southworth, H. Al Daas}

\title{Two-level convergence of Algebraic Multigrid with Overlapping Smoothers and Spectral Coarse Grids%
\thanks{\funding{OAK and BSS were supported by the Laboratory Directed Research and Development program of Los Alamos National Laboratory, under project number 20240261ER. HAD was supported by the Fusion Futures Programme. As announced by the UK Government in October 2023, Fusion Futures aims to provide holistic support for the development of the fusion sector. Los Alamos National Laboratory report number LA-UR-26-24967.}}}

\author{
Oliver A. Krzysik\thanks{Theoretical Division, Los Alamos National Laboratory
(\email{okrzysik@lanl.gov}), \url{https://orcid.org/0000-0001-7880-6512}.}
\and
Ben S. Southworth\thanks{Theoretical Division, Los Alamos National Laboratory
(\email{southworth@lanl.gov}), \url{https://orcid.org/0000-0002-0283-4928}.}
\and
Hussam Al Daas\thanks{Science and Technology Facilities Council, Scientific Computing Department, Rutherford Appleton Laboratory, UK (\email{hussam.al-daas@stfc.ac.uk}), \url{https://orcid.org/0000-0001-9355-4042}.}
}

\newsiamremark{remark}{Remark}
\newsiamremark{hypothesis}{Hypothesis}
\newsiamremark{assumption}{Assumption}
\newsiamremark{fact}{Fact}
\newsiamthm{claim}{Claim}

\crefname{hypothesis}{Hypothesis}{Hypotheses}
\Crefname{hypothesis}{Hypothesis}{Hypotheses}
\crefname{assumption}{Assumption}{Assumptions}
\Crefname{assumption}{Assumption}{Assumptions}
\crefname{fact}{Fact}{Facts}
\Crefname{fact}{Fact}{Facts}
\crefname{claim}{Claim}{Claims}
\Crefname{claim}{Claim}{Claims}

\crefname{section}{Section}{Sections}
\Crefname{section}{Section}{Sections}
\crefname{subsection}{Section}{Sections}
\Crefname{subsection}{Section}{Sections}

\newcolumntype{P}[1]{>{\raggedright\arraybackslash}p{#1}}

\allowdisplaybreaks

\newcommand{\Hdiv}{H(\operatorname{div})}

\newcommand{\wt}[1]{\widetilde{#1}}

\DeclareFontFamily{U}{mathx}{}
\DeclareFontShape{U}{mathx}{m}{n}{<-> mathx10}{}
\DeclareSymbolFont{mathx}{U}{mathx}{m}{n}
\DeclareMathAccent{\wc}{0}{mathx}{"71}
\DeclareMathAccent{\wbar}{0}{mathx}{"73}

\DeclareMathOperator*{\argmin}{arg\,min}
\DeclareMathOperator*{\argmax}{arg\,max}

\DeclareMathOperator{\range}{range}

\begin{document}
\maketitle

\begin{abstract}
We recently developed the least-squares algebraic-multigrid domain-decomposition (LS-AMG-DD) solver as an algebraic multilevel method for sparse symmetric
positive definite matrices that admit a Gram representation \(A=G^{\top}G\) \cite{southworth2026lsamgdd}.
Many problem classes admit such structure, including many conforming finite-element discretizations.
The solver constructs coarse spaces from local eigenproblems on nonoverlapping, algebraic aggregates and uses Schwarz-type smoothers on the induced overlapping subdomains.  
This paper develops a novel two-level convergence theory for this solver.
Our theory shows that the solver's coarse space satisfies a weak approximation property in a norm induced by an aggregate-wise block-Jacobi smoother, and moreover, that the corresponding approximation constant is bounded by a user-controlled local spectral cutoff threshold.
We combine this approximation property with standard sharp theory for multiplicative two-level cycles. The resulting two-level bound is cleanly factored by the cutoff threshold and a smoother norm-comparison constant; we derive explicit bounds for this constant for block Jacobi and overlapping additive Schwarz smoothers.  
We also develop a new convergence bound for additive Schwarz methods in terms of a trivially computable constant that is bounded above by the coloring constant. 
Numerical experiments on scalar 
\(H^1\), vector \(H(\operatorname{div})\), and vector \(H(\operatorname{curl})\) finite-element problems provide supporting evidence for the theory, including evidence for the solver's insensitivity to mesh refinement and polynomial degree.
\end{abstract}

\begin{keywords}
Algebraic multigrid, domain decomposition, spectral coarse spaces, overlapping Schwarz, conforming finite elements, convergence theory, coloring constant
\end{keywords}

\begin{MSCcodes}
65F08, %
65N30, %
65N55 %
\end{MSCcodes}

\section{Introduction}

Algebraic multigrid (AMG) and domain decomposition (DD) are two standard frameworks for solving  large sparse symmetric positive definite (SPD)
systems
\cite{vanek1996sa,xu2017amg,SmithBjorstadGropp2004DD,ToselliWidlund2005DDM}.
AMG methods are largely built around an assumption of algebraically smooth error varying slowly in the direction of strong connections, and almost always consist of relatively simple pointwise smoothers and algebraic coarse-space construction. AMG methods are especially effective for many scalar elliptic problems, where this assumption is naturally satisfied, but are not applicable as black-box solvers for most
conforming \(H(\operatorname{div})\) and \(H(\operatorname{curl})\)
discretizations, and can fail on extreme anisotropy as well \cite{southworth2026lsamgdd,wimmer2024fast}.
Specialized AMG solvers also exist for \(H(\operatorname{div})\) and \(H(\operatorname{curl})\) in the form of ADS and AMS, respectively, that are effective on (at least) isotropic problems, but are less black-box than classical AMG in that they require detailed problem information \cite{kolev2009ams,kolev2012ads}. Very recent work has also considered specialized truly algebraic AMG solvers specifically for two-dimensional \(H(\operatorname{curl})\) problems \cite{shen2026nodal}.
In contrast to AMG, spectral DD methods target robustness by using block-based overlapping Schwarz smoothers on large subdomains, while enriching coarse spaces with local
generalized eigenvectors  \cite{SpillaneEtAl2014GenEO,aldaas2019coarsespaces,aldaas2021multilevel,aldaas2022normaleq}.
Typically, such DD methods are inherently two-level.

In \cite{southworth2026lsamgdd} we introduced the least-squares algebraic-multigrid domain-decomposition (LS-AMG-DD) solver
 to combine concepts from multilevel AMG and two-level spectral DD.  
It is
an algebraic multilevel method for sparse SPD systems that admit a Gram, or
least-squares, representation \(A=G^{\top}G\), using algebraic aggregation,
overlapping Schwarz smoothing, and local generalized eigenproblems built from
Gram-induced symmetric positive semidefinite (SPSD) splittings.
In our companion work \cite{krzysik2026conforming} we show how the Gram input
required by LS-AMG-DD arises naturally in the context of conforming finite
element discretizations, proving an equivalence between locally supported Gram representations and local SPSD splittings. 
We also show robustness of the solver across the challenging problems of grad--div in \(\Hdiv\), anisotropic hyperdiffusion in $H^2$, and linear elasticity in vector $H^1$, particularly in difficult parameter regimes where standard black-box AMG methods outright fail.  
Thus LS-AMG-DD provides an AMG-style multilevel algebraic route
for problem classes where standard black-box AMG is usually not effective. The overarching goal of the solver is AMG-like scalability and low operator complexity when viable, while gaining the robustness of two-level spectral DD. What is absent from \cite{southworth2026lsamgdd,krzysik2026conforming}, and what we address in the present paper, is a convergence theory that explains the quality of the
computed LS-AMG-DD coarse space and the resulting two-level method.

Since LS-AMG-DD blends ingredients from both DD and AMG, there are two natural analytical frameworks one could use to analyze the method. 
On the DD side, two-level Schwarz methods and spectral/adaptive coarse spaces are typically analyzed through abstract Schwarz or fictitious-space type arguments together with stable decomposition estimates; see, for example,
\cite{SmithBjorstadGropp2004DD,ToselliWidlund2005DDM,Xu1992SubspaceCorrection,GalvisEfendiev2010ReducedCoarse,NatafXiangDoleanSpillane2011DtN,SpillaneEtAl2014GenEO,HeinleinKlawonnKnepperRheinbach2019AdaptiveGDSW,aldaas2022normaleq}. This is also typically done in an additive correction framework rather than multiplicative coupling of smoothing and coarse-grid correction (although not always). 
In particular, the use of local generalized eigenvalue problems to define robust coarse spaces is widely used and analyzed in spectral DD literature.
In contrast, two-level multigrid is analyzed as a multiplicative method and based on certain approximation properties that measure the interaction/complementarity of interpolation and smoothing \cite{falgout2005twogrid,maclachlan2014theoretical,xu2017amg,vassilevski2010lecturenotes}. For AMG, necessary, sufficient and tight two-level convergence bounds have been developed for abstract interpolation and smoothing operators \cite{falgout2004generalizing,falgout2005twogrid}. However, this tight analytical framework has not yet been applied to overlapping smoothers and coarse grids built from local spectral problems.

Thus here we analyze LS-AMG-DD in an AMG approximation-property setting, showing that the spectral-based interpolation used in LS-AMG-DD admits the weak approximation property (WAP) needed by AMG theory \cite{falgout2005twogrid,maclachlan2014theoretical,xu2017amg,vassilevski2010lecturenotes}.
Our developments use the traditional multigrid concept of an approximation property, which we combine with subspace correction theory to obtain two-level explicit convergence
bounds for LS-AMG-DD.
Our main result is that the LS-AMG-DD interpolation operator satisfies a WAP whose constant is controlled by the realized (user-defined) local eigenvalue cutoff, which then leads to sharp two-level theory. 

The remainder of the paper is organized as follows. \Cref{sec:prelim} provides background on LS-AMG-DD and the two-level convergence framework we build on.
\Cref{sec:two_level} develops the key convergence theory, with \cref{sec:smoothers} providing concrete estimates for certain smoothers and matrix classes.
In these sections we also present a new convergence bound for additive Schwarz in the context of Gram matrices, which is trivially computable, and bounded above (sometimes strictly) by the coloring constant. Furthermore, we illustrate the nuances of using restricted additive Schwarz (RAS) as a smoother, proving that, despite the fact that it performs excellently in many numerical tests, it can diverge on trivial problems, even for SPD \(\mathsf{M}\)-matrices.
\Cref{sec:theory_interpretation} compares the resulting bounds with existing algebraic DD estimates in the literature, showing that our new theory yields smaller condition number bounds than existing theory for closely related, but different, methods.
Supporting numerical results on challenging finite-element problems are presented in
\cref{sec:numerics}, followed by conclusions in \cref{sec:conclusion}.

\section{Preliminaries}
\label{sec:prelim}

\subsection{Background on LS-AMG-DD}
\label{sec:ls_amg_dd_background}

We now introduce two-level LS-AMG-DD objects relevant for our analysis; we point the reader to \cite{southworth2026lsamgdd,krzysik2026conforming} for more comprehensive descriptions of the method, including its multilevel formulation. 
Let
\(A\in\mathbb R^{n\times n}\) be SPD with Gram representation
\(A=G^{\top}G\), where \(G\in\mathbb R^{m\times n}\).  The integer \(n\) is
the number of degrees of freedom (DOFs), while \(m \geq n\) is the number of equations represented by rows of $G$.  
Since \(A\) is SPD, \(G\) has full column rank.  We
write the columns of \(G^{\top}\) as
\(G^{\top}=[\bm{g}_1,\dots,\bm{g}_m]\), for \(\bm{g}_j \in \mathbb{R}^n\), so that
\(A=\sum_{j=1}^m \bm{g}_j\bm{g}_j^{\top}\). 
The LS-AMG-DD solver takes the form of a standard multiplicative two-level method, so that it has the error-propagation operator
\begin{equation}
\label{eq:ETG_lsdd}
  E_{\mathrm{TG}}
  =
  \bigl(I-M^{-\top}A\bigr)
  \bigl(I-PA_c^{-1}P^{\top}A\bigr)
  \bigl(I-M^{-1}A\bigr),
  \qquad
  A_c:=P^{\top}AP \in \mathbb{R}^{n_c \times n_c}.
\end{equation}
Here \(P \colon \mathbb{R}^{n_c} \to \mathbb{R}^n\) is the interpolation
operator outlined below, and \(M^{-1}\approx A^{-1}\) is the smoothing matrix
associated with a Schwarz-type iteration.  Crucial to the definition of both $P$ and $M^{-1}$ are several DOF partitions introduced below.

The DOFs $\{ 1, \ldots, n\}$ are partitioned into \(n_{\rm agg}\) disjoint aggregates
\begin{align} \label{eq:agg_def}
\{\omega_i\}_{i=1}^{n_{\rm agg}},
\qquad
\bigcup_{i=1}^{n_{\rm agg}}\omega_i = \{1,\dots,n\},
\qquad
\omega_i \cap \omega_k = \emptyset
\quad\text{for } i\neq k,
\end{align}
which are constructed algebraically from the matrix \(A\).  
Next, overlapping $\{ \Omega_i \}_{i = 1}^{n_{\rm agg}}$ and interface $\{ \Gamma_i \}_{i = 1}^{n_{\rm agg}}$ DOF sets are induced from \(\{\omega_i\}_{i=1}^{n_{\rm agg}}\) and the matrix $G$.
For each aggregate, let
\(
\mathfrak{R}_{\omega_i}:=\{j:\operatorname{supp}(\bm{g}^\top_j)\cap\omega_i\ne\emptyset\} \subseteq \{1, \ldots, m \}\)
be the row indices of $G$ whose column support touches aggregate \(\omega_i\).
The associated overlap and interface DOF sets are then defined as
\begin{align} \label{eq:overlap_def}
  \Omega_i
  :=
  \bigcup_{j\in\mathfrak{R}_{\omega_i}}\operatorname{supp}(\bm{g}_j^{\top}),
  \qquad
  \Gamma_i:=\Omega_i\setminus\omega_i .
\end{align}

We use \(R_{\omega_i}\), \(R_{\Gamma_i}\), and \(R_{\Omega_i}\) to denote the Boolean restriction operator from the global DOF set $\{ 1, \ldots, n \}$ onto the aggregate, interface, and
overlap DOF sets, respectively.  
Note that for any DOF index set \(X\subseteq\{1,\dots,n\}\), \(R_X:=I_n(X,:)\) denotes the restriction to the coordinates indexed by \(X\).
On \(\Omega_i = \omega_i \cup \Gamma_i\) we define the partition-of-unity matrix
\(
  D_{\Omega_i}
  :=
  \begin{bsmallmatrix}
    I_{|\omega_i|} & 0\\
    0 & 0_{|\Gamma_i|}
  \end{bsmallmatrix},
\)
supposing DOFs are ordered first by aggregate and then interface.
Since the aggregates form a disjoint partition,
\(
  \sum_{i=1}^{n_{\rm agg}}
  R_{\Omega_i}^{\top}D_{\Omega_i}R_{\Omega_i}
  =
  I .
\)

The smoothing iteration in \eqref{eq:ETG_lsdd} is defined by a Schwarz method
using the aggregate sets \eqref{eq:agg_def}, the overlap sets
\eqref{eq:overlap_def}, or both.
For example, in \cite{southworth2026lsamgdd}, \(M^{-1}\) is taken as a RAS method: the residual is restricted to \(\Omega_i\), the principal submatrix of \(A\) on \(\Omega_i\) is inverted onto it, and the resulting local correction is restricted back to \(\omega_i\).
In this paper we develop detailed theory for $M^{-1}$ given by block Jacobi defined over the aggregate set \eqref{eq:agg_def} and overlapping additive Schwarz defined on the overlapping set \eqref{eq:overlap_def}.
However, our most general theory applies to \cref{eq:ETG_lsdd} regardless of the exact form of $M^{-1}$; our numerical tests consider the three aforementioned smoothers as well as overlapping multiplicative Schwarz defined on \eqref{eq:overlap_def}.

\subsubsection{Local eigenvalue problems and interpolation}
\label{subsec:lsamgdd_splitting}

In LS-AMG-DD, the interpolation \(P\) in \eqref{eq:ETG_lsdd} is block diagonal over the aggregate set
\(\{\omega_i\}_{i=1}^{n_{\rm agg}}\), with each diagonal block formed from a
selected subset of eigenvectors from local generalized eigenvalue problems.
To describe these local eigenproblems, we first introduce a local  symmetric SPSD splitting of the matrix \(A\) on the overlaps \(\{\Omega_i\}_{i=1}^{n_{\rm agg}}\).
For each row index \(j\in\{1,\dots,m\}\) of $G$, define its aggregate-multiplicity $\textrm{mult}_{\omega}(j)$ as the number of aggregates it touches:
\(
\textrm{mult}_{\omega}(j)
:=
\bigl|
\{\, i\in\{1,\dots,n_{\rm agg}\} : j\in\mathfrak{R}_{\omega_i} \,\}
\bigr|.
\)
We assume that any identically zero rows of \(G\) have been
removed, and hence every row touches at least one aggregate, so \(\textrm{mult}_{\omega}(j)\ge 1\) for all \(j\).
For each \(\omega_i\), define the diagonal weighting matrix 
\(
W_{\mathfrak{R}_{\omega_i}}
:=
\operatorname{diag}_{j\in\mathfrak{R}_{\omega_i}} \!\bigl(\textrm{mult}_{\omega}(j)^{-1}\bigr),
\)
and the weighted local Gram factor
\(
H_{\Omega_i}
:=
W_{\mathfrak{R}_{\omega_i}}^{1/2} G(\mathfrak{R}_{\omega_i},\Omega_i).
\)
Then, we have the following local SPSD splitting:
\begin{equation}
\label{eq:SPSD_splitting}
A
=
\sum_{i=1}^{n_{\rm agg}}
R_{\Omega_i}^{\top}\,
\widetilde A_{\Omega_i}\,
R_{\Omega_i},
\quad
\widetilde A_{\Omega_i}
:=
H_{\Omega_i}^{\top} H_{\Omega_i}
\succeq 0.
\end{equation}

For each overlap \(\Omega_i=\omega_i\cup\Gamma_i\), denote the associated principal submatrix of $A$ as
\(
A_{\Omega_i}:=R_{\Omega_i}AR_{\Omega_i}^{\top},
\)
and on $\Omega_i$ assume the DOFs are ordered first by aggregate and then interface. 
The local generalized eigenvalue problem on each $\Omega_i$ is then given by
\begin{equation}
\label{eq:gep_local}
D_{\Omega_i}A_{\Omega_i}D_{\Omega_i}\bm u_i
=
\lambda\,\widetilde A_{\Omega_i}\bm u_i
\quad
\iff
\quad
\begin{bmatrix}
A_{\omega_i,\omega_i} & 0\\
0 & 0
\end{bmatrix}
\begin{bmatrix}
\bm u_{\omega_i}\\
\bm u_{\Gamma_i}
\end{bmatrix}
=
\lambda
\begin{bmatrix}
\widetilde A_{\omega_i,\omega_i} & \widetilde A_{\omega_i,\Gamma_i}\\
\widetilde A_{\Gamma_i,\omega_i} & \widetilde A_{\Gamma_i,\Gamma_i}
\end{bmatrix}
\begin{bmatrix}
\bm u_{\omega_i}\\
\bm u_{\Gamma_i}
\end{bmatrix},
\end{equation}
with the equivalence following because \(D_{\Omega_i}\) is the identity on \(\omega_i\) and zero on \(\Gamma_i\).
For finite nonzero eigenvalues, elimination of the interface variables yields the reduced generalized eigenvalue problem
\begin{equation}
\label{eq:gep_reduced}
A_{\omega_i,\omega_i}\bm u_{\omega_i}
=
\lambda\,\widetilde S_i \bm u_{\omega_i},
\qquad
{\rm with}
\qquad
\widetilde S_i
:=
\widetilde A_{\omega_i,\omega_i}
-
\widetilde A_{\omega_i,\Gamma_i}
\bigl(\widetilde A_{\Gamma_i,\Gamma_i}\bigr)^{\dagger}
\widetilde A_{\Gamma_i,\omega_i},
\end{equation}
where $\dagger$ denotes the Moore-Penrose pseudoinverse, and \(\widetilde S_i\) is the Schur complement of \(\widetilde A_{\Omega_i}\) onto \(\omega_i\). 

For each $i \in \{1, \ldots, n_{\rm agg}\}$ we build a local coarse space using the largest generalized eigenvectors from \eqref{eq:gep_reduced} together with the \(+\infty\) modes corresponding to \(\ker(\widetilde S_i)\). 
The global interpolation operator $P$ is then assembled from the collection of all such local coarse spaces, as in the following definition.
The key user-controlled parameter in defining $P$ is the local spectral cutoff threshold $\tau_{\rm cut} \geq 1$, determining which of the local eigenvectors contribute to the local coarse space.
Our theory in \cref{sec:two_level} shows that convergence of the LS-AMG-DD solver is controlled by $\tau_{\rm cut}$.
Note that it can be shown that $A_{\omega_i, \omega_i} \succeq \wt{S}_i \succeq 0$, and hence $\lambda \geq 1$ in \eqref{eq:gep_reduced}, which is why $\tau_{\rm cut} \geq 1$.

\begin{definition}[Interpolation]
\label{def:coarse_space}
Given \(\widetilde S_i\) in \eqref{eq:gep_reduced}, let 
\(
\{ \bm \psi_{i,k} \}_{k \in \mathcal J_i^{\rm ker}}
\)
be a basis for $\ker(\widetilde S_i)$, with 
\(
\mathcal J_i^{\rm ker}
:=
\{1,\ldots,\dim(\ker(\widetilde S_i))\}.
\)
Let
\(
\{(\lambda_{i,k},\bm \varphi_{i,k})\}_{k\in\mathcal J_i^{\rm fin}}
\)
be the finite generalized eigenpairs of the pencil in \eqref{eq:gep_reduced}, with vectors chosen so that
\(\{\bm\psi_{i,k}\}_{k\in\mathcal J_i^{\rm ker}}
\cup
\{\bm\varphi_{i,k}\}_{k\in\mathcal J_i^{\rm fin}}\)
forms a basis of \(\mathbb R^{|\omega_i|}\).
Partition the finite indices as
\(
\mathcal J_i^{\rm fin}=\mathcal J_i^{\rm disc}\cup\mathcal J_i^{\rm keep},
\)
where 
\(
\mathcal J_i^{\rm disc}
:=
\{k\in\mathcal J_i^{\rm fin}:\lambda_{i,k}\le \tau_{\rm cut}\},
\)
and
\(
\mathcal J_i^{\rm keep}
:=
\{k\in\mathcal J_i^{\rm fin}:\lambda_{i,k}> \tau_{\rm cut}\}.
\). 
The retained local coarse space of dimension $n_{c,i} := \dim(\ker(\widetilde S_i)) + |\mathcal J_i^{\rm keep}|$ is then given by the matrix 
\begin{align} \label{eq:Z-def}
    Z_i
    :=
    \begin{bmatrix}
        Z_i^{\rm ker} \;\, Z_i^{\rm fin}
    \end{bmatrix}
    \in\mathbb{R}^{|\omega_i|\times n_{c,i}},
    \quad
    Z_i^{\rm ker}
    :=
    ( \bm \psi_{i,k} )_{k\in\mathcal J_i^{\rm ker}},
    \quad
    Z_i^{\rm fin}
    :=
    ( \bm \varphi_{i,k} )_{k\in\mathcal J_i^{\rm keep}},
\end{align}
The LS-AMG-DD interpolation operator ${P \in \mathbb{R}^{n \times n_c}}$ in \eqref{eq:ETG_lsdd}, where \(n_c=\sum_{i=1}^{n_{\rm agg}} n_{c,i}\), is
\begin{equation}
\label{eq:lsamgdd_interpolation}
P
=
\Bigl[
R_{\omega_1}^{\top}Z_1
\ \cdots\
R_{\omega_{n_{\rm agg}}}^{\top}Z_{n_{\rm agg}}
\Bigr].
\end{equation}
\end{definition}

We assume throughout the estimates in the next sections that
\(\mathcal J_i^{\rm disc}\neq\emptyset\) for each aggregate \(i\), so that each local coarse space is a proper subspace of \(\mathbb R^{|\omega_i|}\).
If this condition fails for some \(i\), then the local interpolation is exact on that aggregate, and the corresponding local estimate is trivial.
For later use, define the realized local discarded-mode constant
\begin{equation}
\label{eq:tau_i_def}
    \tau_i
    :=
    \max_{k\in\mathcal J_i^{\rm disc}}\lambda_{i,k},
    \qquad
    \tau_{\max}:=\max_{1\le i\le n_{\rm agg}}\tau_i .
\end{equation}
Thus \(\tau_i\le \tau_{\rm cut}\) for all \(i\), and
\(\tau_{\max}\le \tau_{\rm cut}\).

\subsection{Background on sharp two-level convergence theory}
\label{sec:two_level_background}

In this section we provide an overview of the framework and setup used next in \cref{sec:two_level}. 
Let \(A\in\mathbb R^{n\times n}\) be SPD, and let \(P\in\mathbb R^{n\times n_c}\) be the interpolation operator, assumed to have full column rank.
For any SPD matrix \(X\), define the \(X\)-orthogonal projector onto \(\range(P)\) by
\(
\Pi_X(P):=P(P^{\top}XP)^{-1}P^{\top}X.
\)
Since \(P\) is fixed, we abbreviate \(\Pi_A:=\Pi_A(P)\).
Then, consider a multiplicative two-level V(1,1) cycle corresponding to the error propagation operator
\begin{equation}
\label{eq:ETG}
E_{\mathrm{TG}}:=(I-M^{-\top}A) (I-\Pi_A) (I-M^{-1}A).
\end{equation}
We assume the corresponding smoother iteration is convergent in the $A$-norm, $\|I- M^{-1}A\|_A<1 \iff M + M^\top - A \succ 0$ \cite[Thm. 3.1]{vassilevski2010lecturenotes}.
In the tight two-level theory presented below, $M$ appears only by proxy through the so-called symmetrized smoother $\wt{M}$, defined by:
\begin{align} \label{eq:Mtilde}
\widetilde M:=M^{\top}(M+M^{\top}-A)^{-1}M
\;
\iff
\;
I - \wt{M}^{-1} A = (I - M^{-1} A)(I - M^{-\top} A).
\end{align}
The classical weak-approximation-property viewpoint in AMG is that two-level convergence is governed by how well the coarse space $\range(P)$ approximates vectors on the $A$-orthogonal coarse-space complement; see, for example, \cite{falgout2004generalizing,falgout2005twogrid,vassilevski2010lecturenotes,maclachlan2014theoretical}.

\begin{definition}[\(\mathcal M\)-weak approximation property]
\label{def:wap_calM}
Given an SPD matrix ${\cal M}$, we say that the coarse space \(\range(P)\) satisfies the \(\mathcal M\)-WAP if there exists a constant \(C_{\mathcal M}>0\) such that, for every
\(\bm v\in\mathbb R^n\),
\begin{equation}
\label{eq:wap_def}
\min_{\bm v_c\in\range(P)} \|\bm v-\bm v_c\|_{\mathcal M}^2
= \|(I - \Pi_{\mathcal M})\bm v\|_{\mathcal M}^2 
\le
C_{\mathcal M}\,\|\bm v\|_A^2.
\end{equation}
We define $\mathcal{W}_{\mathcal M} > 0$ as the smallest constant such that \eqref{eq:wap_def} holds.
\end{definition}

\begin{theorem}[Sharp two-level]
\label{thm:falgout_tg}
For the two-level method \eqref{eq:ETG}, write \(\Pi_{\widetilde M}:=\Pi_{\widetilde M}(P)\) and define
\begin{equation}
\label{eq:falgout_K}
K_{\mathrm{TG}}
:=
\max_{\bm v \in \mathbb{R}^n \setminus\{0\}}
\frac{\|(I-\Pi_{\widetilde M})\bm v\|_{\widetilde M}^2}{\|\bm v\|_A^2}
=
\max_{\bm v\in\range(I-\Pi_A)\setminus\{0\}}
\frac{\|(I-\Pi_{\widetilde M})\bm v\|_{\widetilde M}^2}{\|\bm v\|_A^2}
=
{\cal W}_{\wt M}.
\end{equation}
Then
\(
\|E_{\mathrm{TG}}\|_A=1-\frac{1}{K_{\mathrm{TG}}}.
\)
Moreover, if $\range(P)$ satisfies the $\widetilde M$-weak approximation property with any constant ${\cal C}_{\widetilde M}$, then
\(
\|E_{\mathrm{TG}}\|_A\le 1-1/{\cal C}_{\widetilde M}.
\)
\end{theorem}

Next in \cref{sec:two_level} we present an upper bound for \(K_{\mathrm{TG}}\) in terms of the LS-AMG-DD eigenvalue threshold \(\tau_{\rm cut}\) from \cref{def:coarse_space}.
Although \cref{thm:falgout_tg} is stated in terms of the \(\widetilde M\)-weak approximation property, the local spectral construction in LS-AMG-DD naturally yields control in the aggregate-wise block Jacobi metric \(M_{\omega {\rm J}}\), which arises from assembling a block Jacobi smoother over the aggregates \( \{ \omega_i \}_{i = 1}^{n_{\rm agg}} \) in \eqref{eq:agg_def}, regardless of the actual smoother $M^{-1}$ used by the solver; this is shown in the next section.
To this end, we now define the SPD operator \(M_{\omega {\rm J}}\) and present a convenient representation for the associated induced norm,\footnote{Throughout this paper \(\| \cdot \|_{D_{\Omega_i}A_{\Omega_i}D_{\Omega_i}}\) and \(\| \cdot \|_{\wt{S}_i}\) are only semi-norms because \(D_{\Omega_i}A_{\Omega_i}D_{\Omega_i}\) and \(\wt{S}_i\) are SPSD rather than SPD.
In particular, the nullspace of the former is characterized by vectors on $\Omega_i$ that are supported only in $\Gamma_i$ (see, e.g., \eqref{eq:gep_local}), while the nullspace of the latter is $\ker (\wt{S}_i)$.
}
\begin{equation}
\label{eq:Mjac_def}
M_{\omega {\rm J}}
:=
\sum_{i=1}^{n_{\rm agg}} R_{\omega_i}^{\top}A_{\omega_i,\omega_i}R_{\omega_i},
\quad
\|\bm z\|_{M_{\omega {\rm J}}}^2
=
\sum_{i=1}^{n_{\rm agg}}\|R_{\omega_i}\bm z\|_{A_{\omega_i,\omega_i}}^2
=
\sum_{i=1}^{n_{\rm agg}}\|R_{\Omega_i}\bm z\|_{D_{\Omega_i}A_{\Omega_i}D_{\Omega_i}}^2.
\end{equation}
\section{General two-level theory}
\label{sec:two_level}

In this section we develop a two-level convergence theory based on the sharp two-level convergence theory outlined in \cref{sec:two_level_background}. 
The argument is organized as follows.
In \cref{sec:conv-roadmap} we detail how to map from a WAP in the \(M_{\omega{\rm J}}\)-norm to a two-level convergence result via a sequence of metric transfers. 
Then, in \cref{subsec:two_level_approx_property}, we prove the required \(M_{\omega{\rm J}}\)-approximation property, and in \Cref{subsec:two_level_theorem} we combine these two ingredients to obtain general two-level convergence theory.
In \cref{sec:nonsymmetric_smoother_case} we consider the application of the theory to nonsymmetric smoothers.
A specialized convergence theory for symmetric smoothers is then presented in \cref{sec:symM}, and this is extended to the case of damped symmetric smoothers in \cref{sec:symM_damped}.

\subsection{Two-level convergence from a block-Jacobi approximation property}
\label{sec:conv-roadmap}

Convergence in \cref{thm:falgout_tg} is phrased in terms of the \(\widetilde M\)-weak approximation property. 
In this section we describe how to pass from a  \(M_{\omega{\rm J}}\)-WAP bound to the \(\widetilde M\)-metric, and hence to the two-level constant \(K_{\mathrm{TG}}\). 
Transferring WAP constants between the different metrics naturally introduces a norm-comparison constant ${\cal N}$ as follows.

\begin{lemma}[WAP transfer]
\label{lem:wap_transfer}
Let ${\cal W}_{\cal M}$ be as in \cref{def:wap_calM}, and let \(\mathcal M_1, \mathcal M_2 \in \mathbb{R}^{n \times n}\) be SPD, and define 
\(
\mathcal N_{\mathcal M_1,\mathcal M_2}
:=
\min\{\,c>0:\ \mathcal M_1\preceq c\,\mathcal M_2\,\}.
\)
Then \(\mathcal W_{\mathcal M_1}\le \mathcal N_{\mathcal M_1,\mathcal M_2}\,\mathcal W_{\mathcal M_2}\).
\end{lemma}

\begin{proof}
By definition of \(\mathcal N_{\mathcal M_1,\mathcal M_2}\), one has \(\mathcal M_1\preceq \mathcal N_{\mathcal M_1,\mathcal M_2}\,\mathcal M_2\), and hence \(\|\bm z\|_{\mathcal M_1}^2\le \mathcal N_{\mathcal M_1,\mathcal M_2}\,\|\bm z\|_{\mathcal M_2}^2\) for every \(\bm z\). Therefore, for every \(\bm v\),
\[
\min_{\bm v_c\in\range(P)}\|\bm v-\bm v_c\|_{\mathcal M_1}^2
\le
\mathcal N_{\mathcal M_1,\mathcal M_2}
\min_{\bm v_c\in\range(P)}\|\bm v-\bm v_c\|_{\mathcal M_2}^2
\le
\mathcal N_{\mathcal M_1,\mathcal M_2}\,\mathcal W_{\mathcal M_2}\,\|\bm v\|_A^2.
\]
Taking the max over \(\bm v\neq 0\) completes the proof.
\end{proof}

Note that \(\mathcal N_{\mathcal M_1,\mathcal M_2}\) is also characterized by a maximum generalized eigenvalue.

\begin{lemma}[Norm-comparison constant]
\label{lem:calN-eig-form}
For SPD matrices \(\mathcal M_1, \mathcal M_2 \in \mathbb{R}^{n \times n}\),
\begin{equation}\label{eq:NM1M2}
\mathcal N_{\mathcal M_1,\mathcal M_2}
:=
\min\{\,c>0:\ \mathcal M_1\preceq c\,\mathcal M_2\,\}
=
\lambda_{\max}(\mathcal M_2^{-1}\mathcal M_1)
=
\sup_{\bm x\ne 0}\frac{\bm x^\top \mathcal M_1 \bm x}{\bm x^\top \mathcal M_2 \bm x}.
\end{equation}
\end{lemma}

\begin{proof}
Since \(\mathcal M_2\) is SPD, \(\mathcal M_2^{1/2}\) is invertible. 
Now
\(\mathcal M_1\preceq c\,\mathcal M_2\) if and only if
\(\mathcal M_2^{-1/2}\mathcal M_1\mathcal M_2^{-1/2}\preceq cI\). Since
\(\mathcal M_2^{-1/2}\mathcal M_1\mathcal M_2^{-1/2}\) is SPD, the smallest such \(c\)
is its largest eigenvalue, i.e.
\(
\mathcal N_{\mathcal M_1,\mathcal M_2}
=
\lambda_{\max}(\mathcal M_2^{-1/2}\mathcal M_1\mathcal M_2^{-1/2}).
\)
Finally,
\(\mathcal M_2^{-1}\mathcal M_1\) is similar to
\(\mathcal M_2^{-1/2}\mathcal M_1\mathcal M_2^{-1/2}\), so they have the same
eigenvalues.
\end{proof}

We can now pass from a \(M_{\omega{\rm J}}\)-WAP to the constant $K_{\rm TG}$ from \cref{thm:falgout_tg}.

\begin{lemma}[$K_{\rm TG}$ from \(M_{\omega{\rm J}}\)-WAP]
\label{lem:KTG_direct_from_WAP_MJac}
Assume 
\(
\rho_M:=\|I-M^{-1}A\|_A<1
\). Then
\[
K_{\mathrm{TG}}
=
\mathcal W_{\widetilde M}
\le
\lambda_{\max}(M_{\omega{\rm J}}^{-1}\widetilde M)\,
\mathcal W_{M_{\omega{\rm J}}}
\quad
\Longrightarrow
\quad
\|E_{\mathrm{TG}}\|_A
\le
1-\frac{1}{
\lambda_{\max}( M^{-1}_{\omega{\rm J}} \widetilde M)
}
\frac{1}{\mathcal W_{M_{\omega{\rm J}}}}
.
\]
\end{lemma}

\begin{proof}
By \cref{thm:falgout_tg}, \(K_{\mathrm{TG}}=\mathcal W_{\widetilde M}\). Applying \cref{lem:wap_transfer} with \((\mathcal M_1,\mathcal M_2)=(\widetilde M,M_{\omega{\rm J}})\) gives
\(
\mathcal W_{\widetilde M}
\le
\mathcal N_{\widetilde M,M_{\omega{\rm J}}}\,\mathcal W_{M_{\omega{\rm J}}}
=
\lambda_{\max}(M_{\omega{\rm J}}^{-1}\widetilde M)\,\mathcal W_{M_{\omega{\rm J}}},
\)
where we used the characterization of $\mathcal N_{\widetilde M,M_{\omega{\rm J}}}$ given in \cref{lem:calN-eig-form}.
Substituting the resulting estimate for \(K_{\mathrm{TG}}\) into \(\|E_{\mathrm{TG}}\|_A=1-\frac{1}{K_{\mathrm{TG}}}\) completes the proof.

To apply this result we require that the matrices \((\mathcal M_1,\mathcal M_2)=(\widetilde M,M_{\omega{\rm J}})\) are SPD.
For the case of \(M_{\omega{\rm J}}\) in \eqref{eq:Mjac_def} this is immediate since it is a block-diagonal matrix with SPD diagonal blocks (principal submatrices of SPD $A$).
Under the assumption that \(\rho_M<1 \iff M + M^{\top} - A \succ 0\) \cite[Thm. 3.1]{vassilevski2010lecturenotes}, the symmetrized smoother $\wt{M} = M^\top (M + M^\top - A)^{-1} M$ is SPD and congruent to the SPD matrix \( (M + M^\top - A)^{-1}\).
\end{proof}

\Cref{lem:KTG_direct_from_WAP_MJac} shows that two-level convergence can be bounded by the block Jacobi WAP constant \({\mathcal W_{M_{\omega{\rm J}}}}\), and the norm-comparison factor \(\lambda_{\max}(M_{\omega{\rm J}}^{-1}\widetilde M)\) between the block Jacobi smoother and the symmetrized smoother.
Next in \cref{subsec:two_level_approx_property} we bound the first of these two factors. 
\subsection{A weak approximation property in the block-Jacobi metric}
\label{subsec:two_level_approx_property}

In this section we construct the \(M_{\omega {\rm J}}\)-orthogonal projector \(\Pi_{M_{\omega {\rm J}}}\) onto \(\range(P)\), for \(M_{\omega {\rm J}}\) in \eqref{eq:Mjac_def}, and then prove
\(
\|(I-\Pi_{M_{\omega {\rm J}}})\bm v\|_{M_{\omega {\rm J}}}^2
\le
\tau_{\max}\,\|\bm v\|_A^2,
\,\,
\forall \bm v\in\mathbb R^n.
\)
The argument proceeds in three steps: first, the cutoff yields a local complement inequality on each aggregate (see \cref{sec:jac_cutoff}); second, these local retained components are assembled into the global projector \(\Pi_{M_{\omega {\rm J}}}\) using the block structure of \(P\) (see \cref{sec:jac_projector}); and third, the local estimates are summed to yield the desired result (see \cref{sec:jac_wap}).

\subsubsection{Characterizing the local cutoff threshold}
\label{sec:jac_cutoff}

The first step is to make explicit what the cutoff parameter $\tau_{\rm cut} \geq 1$ in \cref{def:coarse_space} controls on each aggregate through the reduced local generalized eigenproblem \eqref{eq:gep_reduced}.
To this end, let
\(
\mathcal Z_i(\tau_{\rm cut}):=\range(Z_i)\subseteq\mathbb{R}^{|\omega_i|}
\)
denote the selected local coarse space, with $Z_i$ defined in \eqref{eq:Z-def}.
Further, for each $i$, write
\(
A_i:=A_{\omega_i,\omega_i}
\)
and let $\Pi_i$ denote the $A_i$-orthogonal projector onto $\mathcal Z_i(\tau_{\rm cut})$:
\begin{align} \label{eq:Pi_i}
\Pi_i
:=
Z_i\bigl(Z_i^{\top}A_i Z_i\bigr)^{-1} Z_i^{\top}A_i.
\end{align}
Recall the Schur complement $\wt{S}_i$ defined in \eqref{eq:gep_reduced}.

\begin{lemma}[Local complement inequality]
\label{lem:local_complement_ineq}
Let $\tau_i$ be as in \eqref{eq:tau_i_def}.
Then
\begin{equation}
\label{eq:local_complement_ineq}
\bigl\|(I-\Pi_i)\bm v_i\bigr\|_{A_{i}}^2
\le
\tau_i \,\bigl\|(I-\Pi_i)\bm v_i\bigr\|_{\widetilde S_i}^2
\quad
\forall \bm v_i\in\mathbb{R}^{|\omega_i|}.
\end{equation}
Moreover,
\[
\sup_{\bm v_i\in\mathbb{R}^{|\omega_i|} \, :\, (I-\Pi_i)\bm v_i\neq 0}
\frac{\|(I-\Pi_i)\bm v_i\|_{A_{\omega_i,\omega_i}}^2}
{\|(I-\Pi_i)\bm v_i\|_{\widetilde S_i}^2}
=
\tau_i.
\]
\end{lemma}

\begin{proof}
First we recall the local objects specified in \cref{def:coarse_space}.
Specifically, 
\(
Z_i
    :=
    [
        Z_i^{\rm ker} \; Z_i^{\rm fin}
    ],
    \)
    with columns of 
    \(
    Z_i^{\rm ker}
    :=
    ( \bm \psi_{i,k} )_{k\in\mathcal J_i^{\rm ker}}
    \)
being a basis for $\ker(\widetilde S_i)$, and the columns of \(
    Z_i^{\rm fin}
    :=
    ( \bm \varphi_{i,k} )_{k\in\mathcal J_i^{\rm keep}}
\)
being the retained finite eigenvectors of \eqref{eq:gep_reduced}.
Because $A_i$ is SPD and $\widetilde S_i$ is SPSD, the finite generalized eigenvectors of \eqref{eq:gep_reduced} can be chosen to be an $A_i$-orthogonal set.
Specifically, on the finite-eigenvalue subspace, where $\widetilde S_i$ is SPD, we have
\begin{equation}
\label{eq:finite_gep_orthonormality}
\bm \varphi_{i,k}^{\top}\widetilde S_i\bm \varphi_{i,\ell}=\delta_{k\ell},
\qquad
\bm \varphi_{i,k}^{\top}A_i\bm \varphi_{i,\ell}=\lambda_{i,k}\,\delta_{k\ell},
\qquad
k,\ell\in\mathcal J_i^{\rm fin}.
\end{equation}
Moreover, if $\bm \psi\in\ker(\widetilde S_i)$ and $A_i\bm \varphi_{i,k}=\lambda_{i,k}\widetilde S_i\bm \varphi_{i,k}$ with finite $\lambda_{i,k}$, then
\[
\bm \psi^{\top}A_i\bm \varphi_{i,k}
=
\lambda_{i,k}\,\bm \psi^{\top}\widetilde S_i\bm \varphi_{i,k}
=
0,
\]
so the $+\infty$ modes are $A_i$-orthogonal to all finite generalized eigenvectors.

Given a vector $\bm{v}_i$ and the projector $\Pi_i$ in \eqref{eq:Pi_i}, $\Pi_i \bm{v}_i$ is the $A_i$-orthogonal projection of $\bm{v}_i$ onto the retained local space \(
\mathcal Z_i(\tau_{\rm cut}):=\range(Z_i)
\), and hence $(I-\Pi_i)\bm v_i$ is exactly the component of $\bm v_i$ in the span of the discarded finite modes. 
Hence
\(
(I-\Pi_i)\bm v_i
=
\sum_{k\in\mathcal J_i^{\rm disc}} \alpha_k\,\bm \varphi_{i,k},
\)
for some coefficients $\{\alpha_k\}$.
By construction of the cutoff, every such mode in this sum has finite eigenvalue $\lambda_{i,k}\le \tau_i$. Using this along with \eqref{eq:finite_gep_orthonormality}, we obtain
\begin{align*}
&\bigl\|(I-\Pi_i)\bm v_i\bigr\|_{A_i}^2
=
\Bigl(
\sum_{k\in\mathcal J_i^{\rm disc}} \alpha_k \bm \varphi_{i,k}\, ,
\,
\sum_{k\in\mathcal J_i^{\rm disc}} \alpha_k A_i \bm \varphi_{i,k} \Bigr)
\\
&\quad=
\sum_{k\in\mathcal J_i^{\rm disc}}\lambda_{i,k}\,\alpha_k^2
\le
\tau_i\sum_{k\in\mathcal J_i^{\rm disc}}\alpha_k^2
=
\tau_i\,\bigl\|(I-\Pi_i)\bm v_i\bigr\|_{\widetilde S_i}^2.
\end{align*}
This proves the inequality \eqref{eq:local_complement_ineq} for arbitrary $\bm v_i$.
Then note that the result holds with equality for the particular choice $\bm v_i = \bm \varphi_{i, k_{\rm max}}$, where $k_{\max} = \argmax_{k \in {\cal J}_i^{\rm disc}} \lambda_{i,k}$, which proves the supremum characterization.
\end{proof}

\subsubsection{An orthogonal projector onto the range of interpolation}
\label{sec:jac_projector}

The LS-AMG-DD interpolation operator from \cref{def:coarse_space} is 
\(
P
=
\bigl[
R_{\omega_1}^{\top}Z_1
\ \cdots\ 
R_{\omega_{n_{\rm agg}}}^{\top}Z_{n_{\rm agg}}
\bigr].
\)
Because aggregates $\{ \omega_i \}$ are disjoint, $P$ is block-diagonal by aggregate, such that its action can be written as a sum over aggregate-local pieces and its range can naturally be considered as a direct sum of aggregate-local ranges; that is, given a coarse vector $\bm{c}$ partitioned by aggregate as
\(
\bm c=
(
\bm c_1
\,
,
\,
\ldots
\,
,
\,
\bm c_{n_{\rm agg}}
)^\top
\),
one has
\begin{equation}
\label{eq:P_action_blockwise}
P\bm c
=
\sum_{i=1}^{n_{\rm agg}} R_{\omega_i}^{\top} Z_i \bm c_i,
\qquad
\range(P)
=
\bigoplus_{i=1}^{n_{\rm agg}} R_{\omega_i}^{\top}\range(Z_i).
\end{equation}
Next we assemble the retained local components into a global coarse-space operator. 

\begin{lemma}%
\label{lem:Q0_basic}
Given \(M_{\omega {\rm J}}\) in \eqref{eq:Mjac_def} and $\Pi_i$ in \eqref{eq:Pi_i}, define $\Pi_{M_{\omega {\rm J}}}\, : \, \mathbb{R}^n \to \mathbb{R}^n$ by
\begin{equation}
\label{eq:Q0_def}
\Pi_{M_{\omega {\rm J}}}:=\sum_{i=1}^{n_{\rm agg}} R_{\omega_i}^{\top}\Pi_i R_{\omega_i}.
\end{equation}
Then $\Pi_{M_{\omega {\rm J}}}$ is the $M_{\omega {\rm J}}$-orthogonal projector onto $\range(P)$ for $P$ defined by \eqref{eq:lsamgdd_interpolation}.
Moreover, the complement $(I - \Pi_{M_{\omega {\rm J}}})$ on aggregate $\omega_i$ satisfies
\begin{equation}
\label{eq:Q0_local_remainder}
R_{\omega_i}(I-\Pi_{M_{\omega {\rm J}}})
=
(I-\Pi_i)R_{\omega_i},
\qquad
i=1,\dots,n_{\rm agg}.
\end{equation}
\end{lemma}

\begin{proof}
By \eqref{eq:P_action_blockwise}, any vector in $\range(P)$ is obtained by choosing local pieces from $\range(Z_i)$ and assembling them.
Then observe that, by construction of $\Pi_{M_{\omega {\rm J}}}$, we have \(
\Pi_{M_{\omega {\rm J}}}\bm v = \sum_{i=1}^{n_{\rm agg}} R_{\omega_i}^{\top}\Pi_i R_{\omega_i}\bm v \coloneqq
\sum_{i=1}^{n_{\rm agg}} R_{\omega_i}^{\top}\bm u_i,
\)
with local retained pieces given by
\(
\bm u_i=\Pi_i R_{\omega_i}\bm v\in \range(Z_i).
\)
Hence, the assembled vector $\Pi_{M_{\omega {\rm J}}}\bm v$ lies in $\range(P)$.
Conversely, if $\bm p\in\range(P)$, then by \eqref{eq:P_action_blockwise} there exist local coefficient vectors $\{\bm c_i\}$ such that
\(
\bm p=\sum_{i=1}^{n_{\rm agg}} R_{\omega_i}^{\top}Z_i\bm c_i.
\)
Applying $\Pi_{M_{\omega {\rm J}}}$ and using $\Pi_i Z_i=Z_i$ gives
\[
\Pi_{M_{\omega {\rm J}}}\bm p
=
\sum_{i=1}^{n_{\rm agg}} R_{\omega_i}^{\top}\Pi_i Z_i\bm c_i
=
\sum_{i=1}^{n_{\rm agg}} R_{\omega_i}^{\top}Z_i\bm c_i
=
\bm p,
\]
so $\Pi_{M_{\omega {\rm J}}}$ is a projector onto $\range(P)$.

Since each $\Pi_i$ is the $A_i$-orthogonal projector onto $\range(Z_i)$, it is $A_i$-self-adjoint. Expanding using \eqref{eq:Mjac_def} and the definition of \(\Pi_{M_{\omega {\rm J}}}\) and noting $R_{\omega_i}R_{\omega_i}^\top = I$, while $R_{\omega_i}R_{\omega_j}^\top = \mathbf{0}$ for $j\neq i$, for any $\bm x,\bm y\in\mathbb{R}^n$ one has
\begin{align*}
&(\Pi_{M_{\omega {\rm J}}}\bm x)^\top M_{\omega {\rm J}}\bm y
=
\sum_{i=1}^{n_{\rm agg}} (R_{\omega_i}^\top\Pi_i R_{\omega_i}\bm x)^\top
\sum_{j=1}^{n_{\rm agg}} R_{\omega_j}^{\top}A_jR_{\omega_j}\bm y 
=
\sum_{i=1}^{n_{\rm agg}} \bm x^\top R_{\omega_i}^\top \Pi_i^\top A_i R_{\omega_i}\bm y
\\&\quad
=
\sum_{i=1}^{n_{\rm agg}} \bm x^\top R_{\omega_i}^\top A_i \Pi_i R_{\omega_i}\bm y
=
\sum_{i=1}^{n_{\rm agg}} \bm x^\top R_{\omega_i}^\top A_i R_{\omega_i}\sum_{j=1}^{n_{\rm agg}} R_{\omega_j}^\top\Pi_j R_{\omega_j}\bm y
=
\bm x^\top M_{\omega {\rm J}}\Pi_{M_{\omega {\rm J}}}\bm y.
\end{align*}
Combined with the projector property and the identity $\range(\Pi_{M_{\omega {\rm J}}})=\range(P)$ established above, this shows that $\Pi_{M_{\omega {\rm J}}}$ is the $M_{\omega {\rm J}}$-orthogonal projector onto $\range(P)$.
Because the aggregate interiors are disjoint, restricting \eqref{eq:Q0_def} to $\omega_i$ leaves only the $i$th term, yielding
\(
R_{\omega_i}\Pi_{M_{\omega {\rm J}}}=\Pi_iR_{\omega_i}.
\)
Subtracting this from $R_{\omega_i}$ yields \eqref{eq:Q0_local_remainder}.
\end{proof}

\subsubsection{The approximation property}
\label{sec:jac_wap}

In \cref{thm:jacobi_WAP} below we present the desired $M_{\omega {\rm J}}$ WAP.
First, however, the proof requires a variational characterization of the Schur complement $\wt{S}_i$ from \eqref{eq:gep_reduced}; we also recall the weighted Gram factor $H_{\Omega_i}$ that is used to define the local SPSD splitting matrix $\wt{A}_i$ in \eqref{eq:SPSD_splitting}.

\begin{lemma}%
\label{lem:schur_projected_ls}
Write
\(H_{\Omega_i}=[H_{\omega_i} \,\, H_{\Gamma_i}]\) according to the splitting
\(\Omega_i=\omega_i\cup\Gamma_i\) from \eqref{eq:overlap_def}.  
Let \(\Pi_{\Gamma_i}\) be the
Euclidean orthogonal projector onto \(\range(H_{\Gamma_i})\).  Then
\(\widetilde S_i=H_{\omega_i}^{\top}(I-\Pi_{\Gamma_i})H_{\omega_i}\), and for
every \(\bm z_{\omega_i}\in\mathbb R^{|\omega_i|}\),
\[
  \|\bm z_{\omega_i}\|_{\widetilde S_i}^2
  =
  \min_{\bm z_{\Gamma_i}\in\mathbb R^{|\Gamma_i|}}
  \bigl\|
    H_{\omega_i}\bm z_{\omega_i}
    +
    H_{\Gamma_i}\bm z_{\Gamma_i}
  \bigr\|_2^2 .
\]
\end{lemma}

\begin{proof}
Since \(\widetilde A_{\Omega_i}=H_{\Omega_i}^{\top}H_{\Omega_i}\), its
\((\omega_i,\Gamma_i)\)-block structure gives
\[
  \widetilde S_i
  =
  H_{\omega_i}^{\top}
  \Bigl(
    I-
    H_{\Gamma_i}(H_{\Gamma_i}^{\top}H_{\Gamma_i})^\dagger
    H_{\Gamma_i}^{\top}
  \Bigr)
  H_{\omega_i}.
\]
The matrix in parentheses is \(I-\Pi_{\Gamma_i}\), so the first claim follows.
For fixed \(\bm z_{\omega_i}\), set
\(\bm a:=H_{\omega_i}\bm z_{\omega_i}\).  As
\(\bm z_{\Gamma_i}\) varies, the vector
\(H_{\Gamma_i}\bm z_{\Gamma_i}\) ranges over \(\range(H_{\Gamma_i})\).  Hence
\[
  \min_{\bm z_{\Gamma_i}}
  \|\bm a+H_{\Gamma_i}\bm z_{\Gamma_i}\|_2^2
  =
  \min_{\bm y\in\range(H_{\Gamma_i})}
  \|\bm a+\bm y\|_2^2 .
\]
The minimizer is obtained by taking
\(\bm y=-\Pi_{\Gamma_i}\bm a\), where \(\Pi_{\Gamma_i}\) is the Euclidean
orthogonal projector onto \(\range(H_{\Gamma_i})\).  Therefore the minimum is
\(
  \|(I-\Pi_{\Gamma_i})\bm a\|_2^2
  =
  \bm z_{\omega_i}^{\top}
  H_{\omega_i}^{\top}(I-\Pi_{\Gamma_i})H_{\omega_i}
  \bm z_{\omega_i}
  =
  \|\bm z_{\omega_i}\|_{\widetilde S_i}^2 .
\)
\end{proof}

\begin{theorem}[$\tau_{\rm cut}$-controlled $M_{\omega {\rm J}}$-WAP]
\label{thm:jacobi_WAP}
Let $\tau_{\max}$ be as in \eqref{eq:tau_i_def}. 
Then
\begin{equation}
\label{eq:remainder_bound_tau}
\| (I-\Pi_{M_{\omega {\rm J}}})\bm v \|_{M_{\omega {\rm J}}}^2
\le
\tau_{\max}\,\|\bm v\|_A^2
\qquad
\forall \bm v\in\mathbb{R}^n.
\end{equation}
That is, the minimal WAP constant in the block Jacobi norm satisfies \({\cal W}_{M_{\omega{\rm J}}} \leq \tau_{\max}\).
\end{theorem}

\begin{proof}
Fix $i$ and define the local discarded vector
\(
\bm d_i:=R_{\omega_i}(I-\Pi_{M_{\omega {\rm J}}})\bm v=(I-\Pi_i)R_{\omega_i}\bm v,
\)
where the second identity follows from \eqref{eq:Q0_local_remainder}.
Because $D_{\Omega_i}$ is equal to $1$ on $\omega_i$ and $0$ on $\Gamma_i$, one has
\[
\bigl\|R_{\Omega_i}(I-\Pi_{M_{\omega {\rm J}}})\bm v\bigr\|_{D_{\Omega_i}A_{\Omega_i}D_{\Omega_i}}^2
=
\bigl\|R_{\omega_i}(I-\Pi_{M_{\omega {\rm J}}})\bm v\bigr\|_{A_i}^2
\equiv
\|\bm d_i\|_{A_i}^2.
\]
Applying \cref{lem:local_complement_ineq} gives
\(
\|\bm d_i\|_{A_i}^2
\le
\tau_{\max}\,\|\bm d_i\|_{\widetilde S_i}^2.
\)
Next, expand $R_{\omega_i}\bm v$ in the generalized eigenbasis from the proof of \cref{lem:local_complement_ineq} as
\[
R_{\omega_i}\bm v
=
\sum_{k\in\mathcal J_i^{\rm disc}} \alpha_k\bm \varphi_{i,k}
+
\sum_{k\in\mathcal J_i^{\rm keep}} \beta_k\bm \varphi_{i,k}
+
\sum_{\ell=1}^{\dim(\ker(\widetilde S_i))} \gamma_\ell\bm \psi_{i,\ell},
\]
for some coefficients $\{ \alpha_k \}$ that correspond to the discarded finite eigenvectors, while the coefficients $\{\beta_k \}, \{\gamma_{\ell}\}$ correspond to vectors in $\range(Z_i)$.
Since $\Pi_i$ is the $A_i$-orthogonal projector onto the retained space $\range(Z_i)$, it follows that 
\(
\bm d_i =(I-\Pi_i)R_{\omega_i}\bm v
=
\sum_{k\in\mathcal J_i^{\rm disc}} \alpha_k\bm \varphi_{i,k}.
\)
Using \eqref{eq:finite_gep_orthonormality}, we therefore obtain
\[
\|\bm d_i\|_{\widetilde S_i}^2
=
\sum_{k\in\mathcal J_i^{\rm disc}} \alpha_k^2
\le
\sum_{k\in\mathcal J_i^{\rm disc}} \alpha_k^2 + \sum_{k\in\mathcal J_i^{\rm keep}} \beta_k^2
=
\|R_{\omega_i}\bm v\|_{\widetilde S_i}^2.
\]
The last equality follows here because each $\bm \psi_{i,\ell}\in\ker(\widetilde S_i)$ satisfies $\widetilde S_i\bm \psi_{i,\ell}=0$, such that the kernel component contributes neither to the $\widetilde S_i$-seminorm nor to any $\widetilde S_i$-cross terms; using in addition the $\widetilde S_i$-orthonormality of the finite generalized eigenvectors, it follows that
\(
\|R_{\omega_i}\bm v\|_{\widetilde S_i}^2
=
\sum_{k\in\mathcal J_i^{\rm disc}}\alpha_k^2+\sum_{k\in\mathcal J_i^{\rm keep}}\beta_k^2.
\)

Next we apply \cref{lem:schur_projected_ls} with the vector $\bm z_{\omega_i}:=R_{\omega_i}\bm v$ to get
\[
\|R_{\omega_i}\bm v\|_{\widetilde S_i}^2
=
\min_{\bm w_{\Gamma_i}\in\mathbb{R}^{|\Gamma_i|}}
\bigl\|H_{\omega_i}R_{\omega_i}\bm v + H_{\Gamma_i}\bm w_{\Gamma_i}\bigr\|_2^2
\le
\bigl\|H_{\omega_i}R_{\omega_i}\bm v + H_{\Gamma_i}R_{\Gamma_i}\bm v\bigr\|_2^2,
\]
with the last inequality following by taking
\(\bm w_{\Gamma_i}=R_{\Gamma_i}\bm v\) as an admissible, generally nonoptimal,
choice in the minimization.
Now, we have $R_{\Omega_i} \bm{v} = [ R_{\omega_i} \bm{v}, \, R_{\Gamma_i} \bm{v} ]$; further, recall the weighted Gram construction of the local SPSD matrix $\widetilde A_{\Omega_i}=H_{\Omega_i}^{\top}H_{\Omega_i}$ with $H_{\Omega_i}=[H_{\omega_i}\ \ H_{\Gamma_i}]$. As such, the right-hand side above is exactly
\(
\bigl\|H_{\Omega_i}R_{\Omega_i}\bm v\bigr\|_2^2
=
\|R_{\Omega_i}\bm v\|_{\widetilde A_{\Omega_i}}^2.
\)
Combining the above results we obtain
\[
\bigl\|R_{\Omega_i}(I-\Pi_{M_{\omega {\rm J}}})\bm v\bigr\|_{D_{\Omega_i}A_{\Omega_i}D_{\Omega_i}}^2
=
\|\bm d_i\|_{A_i}^2
\le
\tau_i\,\|\bm d_i\|_{\widetilde S_i}^2
\le
\tau_i\,\|R_{\omega_i}\bm v\|_{\widetilde S_i}^2
\le
\tau_i\,\|R_{\Omega_i}\bm v\|_{\widetilde A_{\Omega_i}}^2.
\]
Summing over $i$ and using the definition of the $M_{\omega {\rm J}}$-norm from \eqref{eq:Mjac_def} yields
\begin{align*}
\| (I-\Pi_{M_{\omega {\rm J}}})\bm v \|_{M_{\omega {\rm J}}}^2
=
\sum_{i = 1}^{n_{\rm agg}}
\bigl\|R_{\Omega_i}(I-\Pi_{M_{\omega {\rm J}}})\bm v\bigr\|_{D_{\Omega_i}A_{\Omega_i}D_{\Omega_i}}^2
\leq
\tau_{\rm max}\,
\sum_{i = 1}^{n_{\rm agg}} \|R_{\Omega_i}\bm v\|_{\widetilde A_{\Omega_i}}^2
\\
=
\tau_{\rm max}\,
\sum_{i = 1}^{n_{\rm agg}} \bm{v}^\top R_{\Omega_i}^{\top}\widetilde A_{\Omega_i}R_{\Omega_i} \bm{v}
=
\tau_{\rm max}\,\bm{v}^\top \Bigl( \sum_{i = 1}^{n_{\rm agg}}  R_{\Omega_i}^{\top}\widetilde A_{\Omega_i}R_{\Omega_i} \Bigr) \bm{v}.
\end{align*}
Applying the SPSD splitting identity 
\(
A=\sum_{i=1}^{n_{\rm agg}} R_{\Omega_i}^{\top}\widetilde A_{\Omega_i}R_{\Omega_i}
\)
from \eqref{eq:SPSD_splitting} to the last expression proves \eqref{eq:remainder_bound_tau}.
\end{proof}

\subsection{General two-level convergence}
\label{subsec:two_level_theorem}

We now combine the WAP from \cref{sec:jac_wap} with the metric-transfer results from \cref{sec:conv-roadmap}.

\begin{theorem}[Error propagator bound]
\label{thm:two_level_general}
Let \(P\) be as in \cref{def:coarse_space}, $\tau_{\max}$ as in \eqref{eq:tau_i_def}, and assume 
\(
\rho_M:=\|I-M^{-1}A\|_A < 1.
\)
Then the error propagator \eqref{eq:ETG} satisfies
\begin{align*}
K_{\mathrm{TG}}
\le 
\lambda_{\max}(M^{-1}_{\omega{\rm J}} \widetilde M)\,\tau_{\max}
\quad
\Longrightarrow
\quad
\|E_{\mathrm{TG}}\|_A
\le
1-\frac{1}{\lambda_{\max}(M^{-1}_{\omega{\rm J}} \widetilde M)}
\frac{1}{\tau_{\max}}.
\end{align*}
\end{theorem}

\begin{proof}
By \cref{thm:jacobi_WAP}, the coarse space satisfies the \(M_{\omega{\rm J}}\)-approximation property with constant \(\tau_{\max}\), i.e. \(\mathcal W_{M_{\omega{\rm J}}}\le \tau_{\max}\). Therefore applying \cref{lem:KTG_direct_from_WAP_MJac} gives the first inequality for $K_{\mathrm{TG}}$.
The associated contraction estimate for \(E_{\mathrm{TG}}\) then follows from the identity \(\|E_{\mathrm{TG}}\|_A=1-\frac{1}{K_{\mathrm{TG}}}\) in \cref{thm:falgout_tg}.
\end{proof}

It is interesting to note that the multigrid literature often considers convergence of the standalone fixed-point iteration, and hence considers the norm of the associated error propagator. 
On the other hand, convergence-style bounds in DD literature are often presented with respect to the condition number of the preconditioned operator. These two viewpoints are not equivalent, but they can be linked in one direction \cite{xu1992subspace}, as in the following exposition; see, e.g., \cite{xu2022convergence} for related $K_{\rm TG}$-to-condition-number treatments.

\begin{corollary}[Condition-number bound]
\label{cor:two_lvl_cond}
Write the two-level error propagator as
\(
E_{\mathrm{TG}}=I-B_{\mathrm{TG}}^{-1} A,
\)
such that \(B_{\mathrm{TG}}^{-1}\) is the two-level preconditioner. 
Then the condition number of the preconditioned operator satisfies
\begin{equation}
\label{eq:condition_number_bound_chi_tau}
\kappa(B_{\mathrm{TG}}^{-1} A)
\le
K_{\rm TG}
\le
\lambda_{\max}(M^{-1}_{\omega{\rm J}} \widetilde M)\,\tau_{\max}
\end{equation}
\end{corollary}

\begin{proof}
Let \(T:=I-M^{-1}A\). Denote \(X^{*A}\) as the \(A\)-adjoint of an operator \(X\), which is defined by
\((X\bm u,\bm v)_A=(\bm u,X^{*A}\bm v)_A\) for all \(\bm u,\bm v\). 
This yields
\(X^{*A}=A^{-1}X^\top A\).  Computing the \(A\)-adjoint of \(T\) using \(T^\top=I-AM^{-\top}\), we obtain \(T^{*A}=I-M^{-\top}A\).  Therefore the two-level error propagator \eqref{eq:ETG} can be written as
\(
  E_{\mathrm{TG}}
  =
  T^{*A}(I-\Pi_A)T .
\)
Since \(\Pi_A\) is the \(A\)-orthogonal projector onto \(\range(P)\), the
operator \(I-\Pi_A\) is also an \(A\)-orthogonal projector.  Hence, for all
\(\bm v\),
\[
  (E_{\mathrm{TG}}\bm v,\bm v)_A
  =
  ((I-\Pi_A)T\bm v,T\bm v)_A
  =
  \|(I-\Pi_A)T\bm v\|_A^2
  \ge 0 .
\]
Thus \(E_{\mathrm{TG}}\) is \(A\)-self-adjoint and \(A\)-positive semi-definite.

It follows that $B_{\mathrm{TG}}^{-1}A$ is also $A$-self-adjoint, in which case we can take an eigenvalue decomposition in the $A$-inner product, $B_{\mathrm{TG}}^{-1}A\bm w_\ell = \lambda_\ell\bm w_\ell$, for $A$-orthogonal eigenvectors $\{\bm w_\ell\}$. Then any $\bm v$ can be written in the $\{\bm w_\ell\}$ basis, and 
\[
    ( \bm v,\bm v )_A \,
    \min_\ell \lambda_\ell 
    \leq 
    ( B_{\mathrm{TG}}^{-1}A\bm v, \bm v )_A 
    \leq 
    ( \bm v,\bm v )_A \, \max_\ell \lambda_\ell .
\]
Note that $E_{\mathrm{TG}}$ is diagonalized under the same basis with eigenvalues $\{1 - \lambda_\ell\}$. Let $\|E_{\mathrm{TG}}\|_A = q < 1$ and recall $E_{\mathrm{TG}}$ is $A$-SPSD. Then $0 \leq 1-\lambda_\ell \leq q$ for all $\ell$, implying $\lambda_\ell \in [1-q,1]$ for all $\ell$. We then have
\(
(1-q)(A\bm v,\bm v)
\le
(AB_{\mathrm{TG}}^{-1}A\bm v,\bm v)
\le
(A\bm v,\bm v)
\)
for all \(\bm v\).  By \cite[Lemma~2.1]{Xu1992SubspaceCorrection},
\(
\kappa(B_{\mathrm{TG}}^{-1}A)\le (1-q)^{-1}.
\)
Applying $\| E_{\rm TG} \| = 1 - 1/K_{\rm TG}$ and then the bound from \cref{thm:two_level_general} completes the proof.

\end{proof}
\subsection{The nonsymmetric smoother case}
\label{sec:nonsymmetric_smoother_case}

Much of the remaining convergence theory in this paper specializes to the case of
symmetric smoothers \(M^{-1}\).
Before this, however, we illustrate the added difficulty associated with analyzing
nonsymmetric smoothers by considering the admissibility of the \(A\)-norm
contractivity assumption required by \cref{thm:falgout_tg}.
Recall from \cref{sec:two_level_background} that for a nonsymmetric smoother with
error propagator \(I-M^{-1}A\), the required assumption is
\(
    \|I-M^{-1}A\|_A < 1,
\)
or, equivalently,
\(
    M+M^\top-A\succ0
\).

\Cref{lem:Q_ms_A_contracting} below shows this property holds for exact multiplicative Schwarz over a
generic coordinate cover
\(\{\mathcal Q_i\}_{i=1}^{n_{\mathcal Q}}\) applied to SPD $A$.
The sets \(\mathcal Q_i\) need only cover the DOFs, but they need not be disjoint.
Thus choosing the nonoverlapping aggregate set, \(\mathcal Q_i=\omega_i\), gives block
Gauss--Seidel, while choosing
\(\mathcal Q_i=\Omega_i\) gives the overlapping multiplicative Schwarz smoother.

\begin{lemma}
\label{lem:Q_ms_A_contracting}
Let \(A=A^\top\succ0\), and let
\(\{\mathcal Q_i\}_{i=1}^{n_{\mathcal Q}}\) be nonempty subsets of
\(\{1,\ldots,n\}\) whose union is \(\{1,\ldots,n\}\).
For \(A_{\mathcal Q_i}:=R_{\mathcal Q_i}AR_{\mathcal Q_i}^{\top}\), define
\[
    P_i
    :=
    R_{\mathcal Q_i}^{\top}
    A_{\mathcal Q_i}^{-1}
    R_{\mathcal Q_i}A .
\]
Let \(M_{\mathcal Q{\rm MS}}^{-1}\) denote the preconditioner associated with one
forward exact multiplicative Schwarz sweep over
\(\{\mathcal Q_i\}_{i=1}^{n_{\mathcal Q}}\), i.e.,
\(
    I-M_{\mathcal Q{\rm MS}}^{-1}A
    =
    (I-P_{n_{\mathcal Q}})\cdots(I-P_1).
\)
Then
\[
    \bigl\| I-M_{\mathcal Q{\rm MS}}^{-1}A \bigr\|_A < 1 .
\]
\end{lemma}

\begin{proof}
For each \(i\), \(A_{\mathcal Q_i}\) is SPD, since it is a principal
submatrix of \(A\). Moreover, \(P_i\) is the \(A\)-orthogonal projection onto
\(\range(R_{\mathcal Q_i}^{\top})\).
The covering assumption implies
\(
    \sum_{i=1}^{n_{\mathcal Q}}
    \range(R_{\mathcal Q_i}^{\top})
    =
    \mathbb R^n .
\)
Thus the decomposition assumption for exact successive subspace correction described by Xu--Zikatanov \cite{xu2017amg} 
is satisfied. Applying \cite[Theorem~4.5]{xu2017amg} gives
\(
    \bigl\|I-M_{\mathcal Q{\rm MS}}^{-1}A\bigr\|_A^2
    =
    1-\frac{1}{1+c_0},
\)
where \(c_0<\infty\) because the covering assumption gives the required
finite-dimensional subspace decomposition. Hence
\(
    \bigl\|I-M_{\mathcal Q{\rm MS}}^{-1}A\bigr\|_A<1 .
\)
\end{proof}

The projection argument above is specific to multiplicative Schwarz. We now consider undamped
Boolean-PoU RAS, which is defined by the preconditioner \cite{cai1999restricted,southworth2026lsamgdd}
\begin{align} \label{eq:RAS_def}
    M_{\Omega{\rm RAS}}^{-1}
    :=
    \sum_{i=1}^{n_{\rm agg}}
    R_{\Omega_i}^{\top}D_iA_{\Omega_i}^{-1}R_{\Omega_i},
    \quad
    A_{\Omega_i} := R_{\Omega_i} A R_{\Omega_i}^\top.
\end{align}
It turns out that \(A=A^\top\succ0\) alone does not imply the same \(A\)-norm
contractivity.  The next lemma gives a three-DOF SPD \(\mathsf M\)-matrix
counterexample.\footnote{See \cref{sec:SPDDM} for the definition of an \(\mathsf{M}\)-matrix.}
In this sense, the smoother is even weaker than undamped block Jacobi (or even pointwise Jacobi, for that matter), which is unconditionally convergent for \(\mathsf M\)-matrices (see \cref{lem:A_le_2_MomegaJ}).

\begin{lemma}
\label{lem:ras_symmetrization_counterexample}
Define the SPD \(\mathsf{M}\)-matrix
\[
    A
    :=
    \begin{bmatrix}
    1 & -3/5 & 0 \\
    -3/5 & 1 & -3/5 \\
    0 & -3/5 & 1
    \end{bmatrix}.
\]
Let the nonoverlapping aggregates be
\(\omega_1=\{1\}\), \(\omega_2=\{2\}\), and \(\omega_3=\{3\}\), with one-layer
overlaps
\(\Omega_1=\{1,2\}\), \(\Omega_2=\{1,2,3\}\), and
\(\Omega_3=\{2,3\}\). 
Then \(M_{\Omega{\rm RAS}}+M_{\Omega{\rm RAS}}^\top-A\) is not SPD.
\end{lemma}

\begin{proof}
The eigenvalues of \(A\) are \(\lambda_1 = 1-3\sqrt{2}/5 = 0.151\ldots\), \(\lambda_2 = 1\), and \(\lambda_3 = 1+3\sqrt{2}/5=1.848\ldots\), so \(A\) is SPD. Since its off-diagonal entries are
nonpositive, \(A\) is an SPD \(\mathsf M\)-matrix.

For the stated overlaps and Boolean PoU matrices, direct calculation then gives
\[
    M_{\Omega{\rm RAS}}
    +
    M_{\Omega{\rm RAS}}^\top
    -
    A
    =
    \frac1{175}
    \begin{bmatrix}
    337 & -240 & 288\\
    -240 & 175 & -240\\
    288 & -240 & 337
    \end{bmatrix}.
\]
The eigenvalues of this matrix are
\[
\lambda_{1} = \frac{16}{7} - \frac{3 \sqrt{737}}{35} = -0.041\ldots,
\quad
\lambda_{2} = \frac{7}{25} = 0.28,
\quad
\lambda_{3} = \frac{16}{7} + \frac{3 \sqrt{737}}{35} 
= 
4.612\ldots
\]
Since $\lambda_1 < 0$, the matrix \(M_{\Omega{\rm RAS}}+M_{\Omega{\rm RAS}}^\top-A\) is not SPD.
\end{proof}

\begin{corollary}
\label{cor:ras_not_automatically_A_contracting}
Let \(M_{\Omega{\rm RAS}}^{-1}\) be as in \eqref{eq:RAS_def}.
Then, the condition \(A=A^\top\succ0\) does not imply
\[
    \bigl\| I-M_{\Omega{\rm RAS}}^{-1}A \bigr\|_A < 1 .
\]
Moreover, this implication does not even hold within the class of SPD \(\mathsf{M}\)-matrices $A$.
\end{corollary}

Recall that $A$-norm contractivity of the smoother is equivalent to $M + M^\top - A \succ 0$; since $A$ is SPD by assumption, a necessary condition of this requirement is that $M + M^\top \succ 0$.
Or, equivalently, that $M^{-1} + M^{-\top} \succ 0$, because $M^{-1} + M^{-\top} = M^{-1}( M + M^\top ) M^{-\top}$ is congruent to $M + M^\top$.
Considering $M^{-1} + M^{-\top}$, we can see that the RAS obstruction is visible at the level of the local decompositions. 
For example, overlapping additive Schwarz,
\(
    M_{\Omega{\rm AS}}^{-1}
    +
    M_{\Omega{\rm AS}}^{-\top}
    =
    2 \sum_i R_{\Omega_i}^{\top}A_{\Omega_i}^{-1}R_{\Omega_i},
\)
is naturally a sum of local SPD terms.  
On the other hand, symmetrized Boolean-PoU RAS, 
\(
    M_{\Omega{\rm RAS}}^{-1} + M_{\Omega{\rm RAS}}^{-\top} 
    =
    \sum_i R_{\Omega_i}^{\top}D_iA_{\Omega_i}^{-1}R_{\Omega_i}
    +
    \sum_i R_{\Omega_i}^{\top}A_{\Omega_i}^{-1}D_iR_{\Omega_i}
    =
    \sum_i R_{\Omega_i}^{\top} (D_iA_{\Omega_i}^{-1} + A_{\Omega_i}^{-1}D_i) R_{\Omega_i}
    ,
\)
is a sum of local saddle-point (and hence, generally, indefinite) terms.
Thus, unlike overlapping additive Schwarz, RAS does not naturally provide a local positive-energy decomposition of its symmetric part. In fact, positivity of
\(M_{\Omega{\rm RAS}}^{-1}+M_{\Omega{\rm RAS}}^{-\top}\), when it holds, is induced by global assembly.
If the symmetrized smoother is not $A$-norm contractive, i.e., \(M_{\Omega{\rm RAS}}+M_{\Omega{\rm RAS}}^{\top} - A\) is not SPD, but \(M_{\Omega{\rm RAS}}^{-1}+M_{\Omega{\rm RAS}}^{-\top}\) is SPD, then one could employ damping to make it contractive. However, it is not clear how one might choose an appropriate damping factor in practice.
\subsection{The symmetric smoother case}
\label{sec:symM}

\Cref{thm:two_level_general} proves two-level convergence for general, convergent smoothers $M^{-1}$ in terms of the constant \(\lambda_{\max}(M_{\omega{\rm J}}^{-1}\widetilde M)\), which measures the cost of transferring from the block-Jacobi metric to the symmetrized smoother metric. 
In general, this constant is difficult to interpret, and even to estimate numerically, recalling the definition of $\wt{M}$ from \eqref{eq:Mtilde}, and that for most smoothers one usually only has access to (the action of) $M^{-1}$ and not $M$.
That being said, if $M^{-1}$ is symmetric, then one can develop a sharp one-sided comparison between \(\widetilde M\) and \(M\) (\cref{lem:Mtilde_from_M}), which can  be used to bound \(\lambda_{\max}(M_{\omega{\rm J}}^{-1}\widetilde M)\) (\cref{lem:transfer_split_bound}).

\begin{lemma}
\label{lem:Mtilde_from_M}
Let \(M\) be symmetric and nonsingular, and define
\(\rho_M:=\|I-M^{-1}A\|_A\).  Then
\[
  \rho_M<1
  \quad\Longleftrightarrow\quad
  0<\lambda_i(M^{-1}A)<2
  \quad\text{for all } i .
\]
In particular, if \(\rho_M<1\), then \(M\) and \(\widetilde M\) are SPD, and from \eqref{eq:NM1M2}
\[
  \mathcal N_{\widetilde M,M}
  =
  \frac{1}{2-\lambda_{\max}(M^{-1}A)}.
\]
\end{lemma}

\begin{proof}
Since \(M\) is symmetric, \(M^{-1}A\) is similar to the symmetric matrix
\(S:=A^{1/2}M^{-1}A^{1/2}\), and therefore has real spectrum.  Moreover,
\(I-M^{-1}A\) is \(A\)-self-adjoint and is similar to \(I-S\).  Hence
\(\rho_M=\|I-S\|_2=\max_i |1-\lambda_i(S)|\).  It follows immediately that
\(\rho_M<1\) if and only if \(0<\lambda_i(S)<2\) for every \(i\). 

Assume now that \(\rho_M<1\).  Then \(S\succ0\), and hence \(M^{-1}\succ0\) by
congruence; therefore \(M\succ0\).  Also \(\lambda_{\max}(M^{-1}A)<2\), so
\(2I-M^{-1/2}AM^{-1/2}\succ0\), equivalently \(2M-A\succ0\).  Thus
\(\widetilde M=M(2M-A)^{-1}M\) is SPD.  Finally,
\[
  M^{-1/2}\widetilde M M^{-1/2}
  =
  \bigl(2I-M^{-1/2}AM^{-1/2}\bigr)^{-1}.
\]
The eigenvalues of \(M^{-1/2}AM^{-1/2}\) are those of \(M^{-1}A\), so
\(\lambda_{\max}(M^{-1}\widetilde M)
=(2-\lambda_{\max}(M^{-1}A))^{-1}\).  By definition of
\(\mathcal N_{\widetilde M,M}\) in \eqref{eq:NM1M2}, this proves the claim.
\end{proof}

\begin{lemma}%
\label{lem:transfer_split_bound}
Let $M$ be symmetric and \(\rho_M:=\|I-M^{-1}A\|_A < 1\). Then
\[
\lambda_{\max}(M_{\omega{\rm J}}^{-1}\widetilde M)
\le
\mathcal N_{\widetilde M,M}\,\mathcal N_{M,M_{\omega{\rm J}}}
=
\frac{\lambda_{\max}(M_{\omega{\rm J}}^{-1} M)}{2-\lambda_{\max}(M^{-1} A)}.
\]
\end{lemma}

\begin{proof}
By \cref{lem:calN-eig-form},
\(
\lambda_{\max}(M_{\omega{\rm J}}^{-1}\widetilde M)
=
\mathcal N_{\widetilde M,M_{\omega{\rm J}}}.
\)
Also, by definition of the comparison constant,
\(
\mathcal N_{\widetilde M,M_{\omega{\rm J}}}
\le
\mathcal N_{\widetilde M,M}\,\mathcal N_{M,M_{\omega{\rm J}}}.
\)
Now use the value of \(\mathcal N_{\widetilde M,M}\) from \cref{lem:Mtilde_from_M}, and replace
\(
\mathcal N_{M,M_{\omega{\rm J}}}
\)
with its maximum eigenvalue characterization. 
\end{proof}

We can now specialize \cref{thm:two_level_general} and \cref{cor:two_lvl_cond} for symmetric $M$.

\begin{corollary}%
\label{cor:two_lvl_cond_sym}
Assume the hypotheses of \cref{thm:two_level_general} and \cref{cor:two_lvl_cond}.
If, in addition, $M$ is symmetric then
\begin{align*}
K_{\mathrm{TG}}
\le 
\frac{\lambda_{\max}({M^{-1}_{\omega{\rm J}}} M)}{
2 - \lambda_{\max}(M^{-1} A)
} \, \tau_{\max}.
\end{align*}
Consequently,
\begin{align*}
\kappa(B_{\mathrm{TG}}^{-1} A)
\le
\frac{\lambda_{\max}({M^{-1}_{\omega{\rm J}}} M)}{
2 - \lambda_{\max}(M^{-1} A)
} \, \tau_{\max},
\quad
\|E_{\mathrm{TG}}\|_A
\le
1-\frac{2 - \lambda_{\max}(M^{-1} A)}{\lambda_{\max}({ M^{-1}_{\omega{\rm J}}} M)}
\frac{1}{\tau_{\max}},
\end{align*}
\end{corollary}

\subsection{The damped symmetric smoother case}
\label{sec:symM_damped}

Damped smoothers are often used in multigrid methods, both to ensure convergence of the smoother, and to improve overall two-level efficiency by making the smoother more efficient on the complement of the coarse space. 
In the present setting, scalar damping naturally arises for both of these reasons. 
Indeed, for
\(M_\zeta:=\zeta^{-1}M\), with \(\zeta\in\mathbb R_+\), the  $K_{\mathrm{TG}}$ upper bound in \cref{cor:two_lvl_cond_sym} separates two distinct factors,
namely the transfer cost \(\lambda_{\max}(M_{\omega{\rm J}}^{-1}M_\zeta)\) and the stability margin \(2-\lambda_{\max}(M_\zeta^{-1}A)\).
Since \(M_\zeta^{-1}A=\zeta M^{-1}A\), damping provides a direct mechanism for controlling both of these.

Applying \cref{thm:two_level_general} to a damped smoother \(M_\zeta:=\zeta^{-1}M\) leads to the bound
\(K_{\mathrm{TG}}^\zeta \le \lambda_{\max}(M_{\omega{\rm J}}^{-1}\widetilde M_\zeta)\tau_{\max}\). However, the dependence of
\(\lambda_{\max}(M_{\omega{\rm J}}^{-1}\widetilde M_\zeta)\) on \(\zeta\) is not clear, and the resulting optimization problem of choosing $\zeta$ to minimize this two-level bound, $\argmin_{\zeta} \lambda_{\max}(M_{\omega{\rm J}}^{-1}\wt{M}_\zeta)$, is not readily solved in closed form.
By contrast, the \(\zeta\)-dependence for the symmetric bound from \cref{cor:two_lvl_cond_sym} is explicit: the transfer factor scales like \(\zeta^{-1}\), while the stability margin becomes \(2-\zeta\,\lambda_{\max}(M^{-1}A)\).

\begin{corollary}[Damped symmetric smoother]
\label{cor:damped_two_level_bound}
Let \(M\) be SPD, and for \(\zeta>0\) define the damped smoother
\(M_\zeta:=\zeta^{-1}M\).
Then
\[
\rho_{M_\zeta}<1
\quad\Longleftrightarrow\quad
0<\zeta<\frac{2}{\lambda_{\max}(M^{-1}A)}.
\]
For every such \(\zeta\), the associated two-grid constant satisfies
\[
K_{\mathrm{TG}}^\zeta
\le
\frac{\lambda_{\max}(M_{\omega{\rm J}}^{-1}M)}
{\zeta\bigl(2-\zeta\,\lambda_{\max}(M^{-1}A)\bigr)}
\tau_{\max}
.
\]
Moreover, the right-hand side is minimized at
\(\zeta_*=\lambda_{\max}(M^{-1}A)^{-1}\), where it takes the value
\(
K_{\mathrm{TG}}^{\zeta_*}
\le
\lambda_{\max}(M_{\omega{\rm J}}^{-1}M)\,\lambda_{\max}(M^{-1}A)
\tau_{\max}
.
\)
\end{corollary}

\begin{proof}
Because \(M\) is SPD, \(M^{-1}A\) is similar to the SPD matrix
\(M^{-1/2}AM^{-1/2}\), so its eigenvalues are positive.
Since \(M_\zeta^{-1}A=\zeta M^{-1}A\), we have
\(\lambda_{\max}(M_\zeta^{-1}A)=\zeta\,\lambda_{\max}(M^{-1}A)\). Thus, by
\cref{lem:Mtilde_from_M}, \(\rho_{M_\zeta}<1\) if and only if
\(\lambda_{\max}(M_\zeta^{-1}A)<2\), which is equivalent to
\(0<\zeta<2/\lambda_{\max}(M^{-1}A)\).

Also, since \(M_\zeta=\zeta^{-1}M\), we have
\(\lambda_{\max}(M_{\omega{\rm J}}^{-1}M_\zeta)
=\zeta^{-1}\lambda_{\max}(M_{\omega{\rm J}}^{-1}M)\).
Applying \cref{thm:two_level_general} with \(M_\zeta\) in place of \(M\) gives
\[
K_{\mathrm{TG}}^\zeta
\le
\lambda_{\max}(M_{\omega{\rm J}}^{-1}\widetilde M_\zeta)\,\tau_{\max}
\le
\frac{\tau_{\max}\,\zeta^{-1}\lambda_{\max}(M_{\omega{\rm J}}^{-1}M)}
{2-\zeta\,\lambda_{\max}(M^{-1}A)}
=
\frac{\tau_{\max}\,\lambda_{\max}(M_{\omega{\rm J}}^{-1}M)}
{\zeta\bigl(2-\zeta\,\lambda_{\max}(M^{-1}A)\bigr)}.
\]
Finally, writing \(\beta:=\lambda_{\max}(M^{-1}A)\), the denominator
\(\zeta(2-\zeta\beta)\) is maximized at \(\zeta=\beta^{-1}\), so the displayed explicit upper
bound is minimized there; at \(\zeta_*=\beta^{-1}\) it equals
\(\tau_{\max}\,\lambda_{\max}(M_{\omega{\rm J}}^{-1}M)\beta\).
\end{proof}

In numerical experiments, not shown here for brevity, we have tended to observe for block Jacobi smoothing that the optimal damping value predicted above \(\zeta_*=\lambda_{\max}(M^{-1}A)^{-1}\) is generally not quite optimal (often underdamping).
This is not something we explore further here, since, as numerical results in \cref{sec:numerics} demonstrate, RAS or overlapping multiplicative Schwarz smoothers tend to yield much faster convergence than block Jacobi in practice.\footnote{For the special case of block Jacobi smoothing,
\(M=M_{\omega{\rm J}}\), one can derive sharper two-level estimates than the
general bounds used here. Optimizing those specialized estimates also gives
more accurate predictions of the optimal damping parameter. We omit this
block-Jacobi-specific refinement because block Jacobi is not the main smoother
targeted by the method and, in the reported experiments, is typically less
effective than RAS and overlapping multiplicative
Schwarz.}

\section{Concrete two-level estimates for certain SPD smoothers}
\label{sec:smoothers}

The WAP from \cref{sec:two_level} controls the coarse space in the
aggregate block-Jacobi metric \(M_{\omega{\rm J}}\).  To turn this into a
two-level convergence estimate for a particular SPD smoother \(M\), one must also
estimate the cost of transferring that particular WAP to the symmetrized smoother norm. In \cref{cor:damped_two_level_bound}
this transfer cost enters through \(\lambda_{\max}(M_{\omega{\rm J}}^{-1}M)\) and
\(\lambda_{\max}(M^{-1}A)\), giving the estimate
\(K_{\mathrm{TG}}\lesssim
\lambda_{\max}(M_{\omega{\rm J}}^{-1}M)\lambda_{\max}(M^{-1}A)\tau_{\max}\).
Bounding these quantities is important because it shows
whether the two-level estimate is stable with respect to problem parameters, such as mesh refinement and polynomial degree in the finite-element setting.  
We note also that 
\(\lambda_{\max}(M^{-1}A)\) determines the relevant damping scaling factor 
\(\zeta_*=\lambda_{\max}(M^{-1}A)^{-1}\). 
This section gives explicit bounds for these smoother-dependent factors for block Jacobi and overlapping additive Schwarz.
Note that in practice, one could cheaply estimate \(\lambda_{\max}(M^{-1}A)^{-1}\) using a few power iterations, and thus have a practical way to choose the local eigenvalue cutoff to target a desired post-setup constant.
This is analogous in spirit to cutoff selection in algebraic spectral DD
methods, where retained modes are chosen according to target condition numbers
and coloring constants
\cite{aldaas2019coarsespaces,aldaas2022normaleq,AlDaas-2023-twolvlsparse}.

\subsection{Block Jacobi and overlapping additive Schwarz smoothers}
\label{sec:smoother-specialization}

We now specialize \cref{cor:damped_two_level_bound} to the two most natural smoothers in the present SPD setting: aggregate-wise block-Jacobi \(M_{\omega{\rm J}}\) from \eqref{eq:Mjac_def} and overlapping additive Schwarz \(M_{\Omega{\rm AS}}\).
The overlapping additive Schwarz smoother is defined with respect to the overlaps $\{ \Omega_i \}$ in \eqref{eq:overlap_def} by
\begin{equation}
\label{eq:MASinv_def}
M_{\Omega{\rm AS}}^{-1}
:=
\sum_{i=1}^{n_{\rm agg}} R_{\Omega_i}^{\top}A_{\Omega_i}^{-1}R_{\Omega_i},
\qquad
A_{\Omega_i}:=R_{\Omega_i} A R_{\Omega_i}^\top.
\end{equation}
We first show that 
\(
\lambda_{\max}(M_{\omega{\rm J}}^{-1}M_{\Omega{\rm AS}})
\le
\lambda_{\max}(M_{\omega{\rm J}}^{-1} M_{\omega{\rm J}}) = 1\)
(see \cref{lem:as_vs_MaggJac} below),
such that we have the estimate \(K_{\mathrm{TG}}\lesssim \lambda_{\max}(M^{-1}A)\tau_{\max}\) for
\(
M \in \{
M_{\omega{\rm J}},\,
M_{\Omega{\rm AS}} \}
\) (see \cref{lem:two_level_concrete} below).
Next, in \cref{sec:AS}, we develop bounds for \( \lambda_{\max}(M^{-1}A)\).

\begin{lemma}%
\label{lem:as_vs_MaggJac}
It holds that
\(
\|\bm z\|_{M_{\Omega{\rm AS}}}^2
\le
\|\bm z\|_{M_{\omega {\rm J}}}^2\,
\forall
\bm z\in\mathbb R^n.
\)
Consequently,
\[
\lambda_{\max}(M_{\omega {\rm J}}^{-1}M_{\Omega{\rm AS}})
=
\sup_{\bm z\ne 0}
\frac{\|\bm z\|_{M_{\Omega{\rm AS}}}^2}{\|\bm z\|_{M_{\omega {\rm J}}}^2}
\le 1.
\]
\end{lemma}

\begin{proof}
Apply \cite[Thm.~4.4]{xu2017amg} with global space \(V=\mathbb R^n\), local spaces \(X_i:=\mathbb R^{|\Omega_i|}\), injections \(R_{\Omega_i}^{\top}:X_i\to\mathbb R^n\), and exact local solvers \(A_{\Omega_i}^{-1}:X_i\to X_i\).\footnote{In the notation of \cite[\S4.3]{xu2017amg}, our \(X_i\), \(R_{\Omega_i}^{\top}\), and \(A_{\Omega_i}^{-1}\) are their \(V_i\), \(\Pi_i\), and \(R_i\), respectively.} Then
\(B_{{\rm psc}}=\sum_{i=1}^{n_{\rm agg}} R_{\Omega_i}^{\top}A_{\Omega_i}^{-1}R_{\Omega_i}=M_{\Omega{\rm AS}}^{-1}\), and therefore
\begin{equation}\label{eq:Bpsc}
\|\bm z\|_{M_{\Omega{\rm AS}}}^2
=
(B_{{\rm psc}}^{-1}\bm z,\bm z)
=
\min_{\bm z=\sum_{i=1}^{n_{\rm agg}} R_{\Omega_i}^{\top}\bm z_i}
\sum_{i=1}^{n_{\rm agg}} \|\bm z_i\|_{A_{\Omega_i}}^2.
\end{equation}
Now consider the decomposition given by
\(
\widehat{\bm z}_i := D_{\Omega_i}R_{\Omega_i}\bm z, \,  i=1,\dots,n_{\rm agg},
\)
where it is easy to confirm that
\(
\sum_{i=1}^{n_{\rm agg}} R_{\Omega_i}^{\top}\widehat{\bm z}_i
=
\sum_{i=1}^{n_{\rm agg}} R_{\Omega_i}^{\top}D_{\Omega_i}R_{\Omega_i}\bm z
=
\bm z.
\)
Thus \((\widehat{\bm z}_1,\dots,\widehat{\bm z}_{n_{\rm agg}})\) is one element of the admissible set in the minimum in \eqref{eq:Bpsc}, in which case the minimum is bounded above by its value at this choice:
\[
\|\bm z\|_{M_{\Omega{\rm AS}}}^2
\le
\sum_{i=1}^{n_{\rm agg}} \|\widehat{\bm z}_i\|_{A_{\Omega_i}}^2
=
\sum_{i=1}^{n_{\rm agg}} \|D_{\Omega_i}R_{\Omega_i}\bm z\|_{A_{\Omega_i}}^2
=
\sum_{i=1}^{n_{\rm agg}} \|R_{\Omega_i}\bm z\|_{D_{\Omega_i}A_{\Omega_i}D_{\Omega_i}}^2
=
\|\bm z\|_{M_{\omega {\rm J}}}^2.
\]
\end{proof}

\begin{lemma}
\label{lem:two_level_concrete}
Consider 
\(
M \in \{
M_{\omega{\rm J}},\,
M_{\Omega{\rm AS}} \}
\),
and define
\(M_\zeta:=\zeta^{-1}M\).
If \(0<\zeta<2/\lambda_{\max}(M^{-1}A)\), then \(\rho_{M_\zeta}<1\) and
\[
K_{\mathrm{TG}}^\zeta
\le
\frac{\tau_{\max}}{\zeta(2-\zeta \lambda_{\max}(M^{-1}A))}.
\]
Moreover, the right-hand side is minimized at \(\zeta = \zeta_*=1/\lambda_{\max}(M^{-1}A)\), yielding
\(
K_{\mathrm{TG}}^{\zeta_*}
\le
 \lambda_{\max}(M^{-1}A) \, \tau_{\max}.
\)
\end{lemma}

\begin{proof}
From \cref{cor:damped_two_level_bound}, for
\(
M \in \{M_{\omega{\rm J}},\,M_{\Omega{\rm AS}}\}
\),
one has
\[
K_{\mathrm{TG}}^\zeta
\le
\frac{\lambda_{\max}(M_{\omega{\rm J}}^{-1}M)\, \tau_{\max}}
{\zeta\bigl(2-\zeta\,\lambda_{\max}(M^{-1}A)\bigr)}
\le
\frac{\tau_{\max}}
{\zeta\bigl(2-\zeta\,\lambda_{\max}(M^{-1}A)\bigr)},
\]
noting that \(\lambda_{\max}(M_{\omega{\rm J}}^{-1}M_{\Omega{\rm AS}})\le 1\) from \cref{lem:as_vs_MaggJac}.
Minimization of the bound follows analogously to \cref{cor:damped_two_level_bound}.

\end{proof}

\subsubsection{Additive Schwarz bounds}
\label{sec:AS}

To explicitly quantify two-level convergence, it thus remains to bound the smoother
factor \(\lambda_{\max}(M^{-1}A)\). Standard additive Schwarz estimates usually
bound this factor by a coloring constant associated with the subdomain
interaction graph. In the present setting, the Gram representation
\(A=G^\top G=\sum_{j=1}^m \bm g_j\bm g_j^\top\) gives a more direct row-wise
bound: the relevant constant $\nu_{\mathcal Q}$ defined in \cref{thm:nu_touching_AS} below counts how many Schwarz subdomains can be touched by a
single row of the global Gram factor \(G \in \mathbb{R}^{m \times n}\).
Note that we state the estimate for a generic subdomain cover
\(\{\mathcal Q_i\}_{i=1}^{n_{\mathcal Q}}\). This avoids separating the
nonoverlapping and overlapping cases, since choosing
\(\mathcal Q_i=\omega_i\) gives the block Jacobi smoother
\(M_{\omega{\rm J}}^{-1}\), while choosing \(\mathcal Q_i=\Omega_i\) gives the
overlapping additive Schwarz smoother \(M_{\Omega{\rm AS}}^{-1}\).\footnote{In \cite{krzysik2026conforming} we show an equivalence between the existence of a matrix $A$ admitting a locally-supported Gram representation and a local SPSD splitting. Therefore, the bound in \cref{thm:nu_touching_AS} can be reformulated for matrices $A$ given by an SPSD splitting instead of $A = G^\top G$.} 

Combining \cref{lem:two_level_concrete} with \cref{thm:nu_touching_AS} gives
\begin{align} \label{eq:K_concrete_normalized}
    K_{\mathrm{TG}}^{\zeta_*}
    \le
    \tau_{\max} \, \nu_{\omega}
    \quad\text{for } M=M_{\omega{\rm J}},
    \qquad
    K_{\mathrm{TG}}^{\zeta_*}
    \le
    \tau_{\max} \, \nu_{\Omega}
    \quad\text{for } M=M_{\Omega{\rm AS}}.
\end{align}
For standard finite-element discretizations these estimates
are mesh independent since the corresponding Gram factors $G$ have bounded and localized row supports \cite{krzysik2026conforming}.

\begin{theorem}
\label{thm:nu_touching_AS}
Let \(G\in\mathbb R^{m\times n}\) have full column rank, set
\(A=G^\top G\), and let \(\{\mathcal Q_i\}_{i=1}^{n_{\mathcal Q}}\) be
nonempty subsets of \(\{1,\ldots,n\}\) whose union is \(\{1,\ldots,n\}\).
For \(A_{\mathcal Q_i}:=R_{\mathcal Q_i}AR_{\mathcal Q_i}^{\top}\), define
the exact additive Schwarz preconditioner
\[
    M_{\mathcal Q{\rm AS}}^{-1}
    :=
    \sum_{i=1}^{n_{\mathcal Q}}
    R_{\mathcal Q_i}^{\top}A_{\mathcal Q_i}^{-1}R_{\mathcal Q_i}.
\]
Writing \(\{\bm g_j^\top\}_{j=1}^m\) for the rows of \(G\), let
\(\mathfrak R_{\mathcal Q_i}:=\{j:\operatorname{supp}(\bm g_j^\top)\cap
\mathcal Q_i\neq\emptyset\}\), let
\(\operatorname{mult}_{\mathcal Q}(j):=
|\{i:j\in\mathfrak R_{\mathcal Q_i}\}|\), and set
\(\nu_{\mathcal Q}:=\max_j\operatorname{mult}_{\mathcal Q}(j)\). Then
\[
    \lambda_{\max}(M_{\mathcal Q{\rm AS}}^{-1}A)
    \le
    \nu_{\mathcal Q}.
\]
\end{theorem}

\begin{proof}
By the variational characterization of exact additive Schwarz, for every \(\bm x\in\mathbb R^n\),\footnote{Here we use \cite[Thm~4.4]{xu2017amg} with
\(V=\mathbb R^n\), \(V_i=\mathbb R^{|\mathcal Q_i|}\),
\(\Pi_i=R_{\mathcal Q_i}^{\top}\), and
\(R_i=A_{\mathcal Q_i}^{-1}\). Then their parallel subspace correction
operator \(B_{\rm psc}=\sum_i\Pi_iR_i\Pi_i^\top\) is
\(M_{\mathcal Q{\rm AS}}^{-1} = \sum_{i=1}
    R_{\mathcal Q_i}^{\top}A_{\mathcal Q_i}^{-1}R_{\mathcal Q_i}\), and their identity
\((B_{\rm psc}^{-1}v,v)=
\min_{\sum_i\Pi_i v_i=v}\sum_i(R_i^{-1}v_i,v_i)\)
gives the displayed variational characterization. The covering condition
\(\bigcup_i\mathcal Q_i=\{1,\ldots,n\}\) is precisely the decomposition
assumption \(V=\sum_i\Pi_iV_i\) from \cite[Sec.~4.3]{xu2017amg}}
\[
\bm x^\top M_{\mathcal Q{\rm AS}}\bm x
=
\min_{\bm x=\sum_{i=1}^{n_{\mathcal Q}} R_{\mathcal Q_i}^{\top}\bm x_i}
\sum_{i=1}^{n_{\mathcal Q}} \|\bm x_i\|_{A_{\mathcal Q_i}}^2.
\]
Now fix any admissible decomposition
\(\bm x=\sum_{i=1}^{n_{\mathcal Q}} R_{\mathcal Q_i}^{\top}\bm x_i\).
For each row index \(j\), since
\(\bm g_j^\top R_{\mathcal Q_i}^\top=0\) whenever
\(j\notin \mathfrak R_{\mathcal Q_i}\), we have
\(
\bm g_j^\top \bm x
=
\sum_{i=1}^{n_{\mathcal Q}}\bm g_j^\top R_{\mathcal Q_i}^{\top}\bm x_i
=
\sum_{i:\, j\in\mathfrak R_{\mathcal Q_i}} \bm g_j^\top R_{\mathcal Q_i}^{\top}\bm x_i.
\)
For a fixed row index \(j\), set
\(S_j:=\{i:\,j\in\mathfrak R_{\mathcal Q_i}\}\). The preceding identity reads
\(\bm g_j^\top\bm x=\sum_{i\in S_j}\bm g_j^\top R_{\mathcal Q_i}^{\top}\bm x_i\),
and \(|S_j|=\operatorname{mult}_{\mathcal Q}(j)\). We now use the Cauchy--Schwarz inequality applied to \(( \boldsymbol{\theta},\mathbf{1} )\) in the form
\(
\big|\sum_{i\in S_j}\theta_i\big|^2
\le
|S_j|\sum_{i\in S_j}|\theta_i|^2 .
\)
Applying this with
\(\theta_i=\bm g_j^\top R_{\mathcal Q_i}^{\top}\bm x_i\) gives
\[
|\bm g_j^\top \bm x|^2
\le
\operatorname{mult}_{\mathcal Q}(j)
\sum_{i\in S_j}
|\bm g_j^\top R_{\mathcal Q_i}^{\top}\bm x_i|^2
\le
\nu_{\mathcal Q}
\sum_{i=1}^{n_{\mathcal Q}}
|\bm g_j^\top R_{\mathcal Q_i}^{\top}\bm x_i|^2 .
\]
Here the last step uses
\(\operatorname{mult}_{\mathcal Q}(j)\le\nu_{\mathcal Q}\) and the fact that
\(\bm g_j^\top R_{\mathcal Q_i}^{\top}=0\) for \(i\notin S_j\).

Summing over \(j\) gives
\[
\bm x^\top A\bm x
=
\bigl(
G \bm{x}, G \bm{x}
\bigr)
=
\sum_{j=1}^m |\bm g_j^\top \bm x|^2
\le
\nu_{\mathcal Q}\sum_{i=1}^{n_{\mathcal Q}}\sum_{j=1}^m |\bm g_j^\top R_{\mathcal Q_i}^{\top}\bm x_i|^2
=
\nu_{\mathcal Q}\sum_{i=1}^{n_{\mathcal Q}} \|\bm x_i\|_{A_{\mathcal Q_i}}^2,
\]
where the last equality follows from
\(
A_{\mathcal Q_i}=R_{\mathcal Q_i}AR_{\mathcal Q_i}^\top
=\sum_{j=1}^m (\bm g_j^\top R_{\mathcal Q_i}^\top)^\top(\bm g_j^\top R_{\mathcal Q_i}^\top).
\)
Since the bound above holds for every admissible decomposition
\(\bm x=\sum_{i=1}^{n_{\mathcal Q}} R_{\mathcal Q_i}^{\top}\bm x_i\), and the
left-hand side is independent of that decomposition, we may take the minimum
over all such tuples on the right to obtain
\[
\bm x^\top A\bm x
\le
\nu_{\mathcal Q}
\min_{\bm x=\sum_{i=1}^{n_{\mathcal Q}} R_{\mathcal Q_i}^{\top}\bm x_i}
\sum_{i=1}^{n_{\mathcal Q}}\|\bm x_i\|_{A_{\mathcal Q_i}}^2
=
\nu_{\mathcal Q}\,\bm x^\top M_{\mathcal Q{\rm AS}}\bm x.
\]
Hence,
\[
\lambda_{\max}(M_{\mathcal Q{\rm AS}}^{-1}A)
=
\sup_{\bm x\ne 0}
\frac{\bm x^\top A\bm x}{\bm x^\top M_{\mathcal Q{\rm AS}}\bm x}
\le
\nu_{\mathcal Q}.
\]
\end{proof}

It is useful to compare \(\nu_{\mathcal Q}\) with the coloring constants that arise in
standard additive Schwarz estimates. Define the row-touching incidence matrix
\(\mathcal E_{\mathcal Q}\in\{0,1\}^{m\times n_{\mathcal Q}}\) by
\[
    (\mathcal E_{\mathcal Q})_{ji}
    =
    \begin{cases}
    1, & \operatorname{supp}(\bm g_j^\top)\cap\mathcal Q_i\neq\emptyset,\\
    0, & \text{otherwise}.
    \end{cases}
\]
Thus the \(j\)th row of \(\mathcal E_{\mathcal Q}\) records which Schwarz
subdomains are touched by the \(j\)th row of \(G\), and
\(\nu_{\mathcal Q}=\max_j\sum_i(\mathcal E_{\mathcal Q})_{ji}
=\|\mathcal E_{\mathcal Q}\|_\infty\).

Let \(\mathcal C_{\mathcal Q}\in\{0,1\}^{n_{\mathcal Q}\times n_{\mathcal Q}}\)
be the off-diagonal Boolean pattern of
\(\mathcal E_{\mathcal Q}^\top\mathcal E_{\mathcal Q}\). Equivalently,
\((\mathcal C_{\mathcal Q})_{ik}=1\) if and only if \(i\neq k\) and there is a
row \(\bm g_j^\top\) whose support intersects both \(\mathcal Q_i\) and
\(\mathcal Q_k\).
For a fixed row \(j\), the contribution of the rank-one matrix
\(\bm g_j\bm g_j^\top\) to the interaction between the two Schwarz subdomains
\(\mathcal Q_i\) and \(\mathcal Q_k\) is the restricted matrix
\(
    R_{\mathcal Q_i}\bm g_j\bm g_j^\top R_{\mathcal Q_k}^{\top}
    =
    (R_{\mathcal Q_i}\bm g_j)(R_{\mathcal Q_k}\bm g_j)^\top .
\)
This restricted outer product is nonzero exactly when
\(\operatorname{supp}(\bm g_j^\top)\) intersects both \(\mathcal Q_i\) and
\(\mathcal Q_k\). Hence \((\mathcal C_{\mathcal Q})_{ik}=1\) exactly when
some row outer product in \(A=G^\top G=\sum_{j=1}^m\bm g_j\bm g_j^\top\)
gives a nonzero restricted contribution between \(\mathcal Q_i\) and
\(\mathcal Q_k\).\footnote{The standard \(A\)-interaction graph for the
Schwarz subspaces has an edge \(i\)--\(k\) when
\(R_{\mathcal Q_i}AR_{\mathcal Q_k}^{\top}\neq0\). The graph
\(\mathcal C_{\mathcal Q}\) is always a supergraph of this graph: if
\((\mathcal C_{\mathcal Q})_{ik}=0\), then no row of \(G\) touches both
\(\mathcal Q_i\) and \(\mathcal Q_k\), and hence
\(R_{\mathcal Q_i}AR_{\mathcal Q_k}^{\top}=0\). Equivalently,
\(R_{\mathcal Q_i}AR_{\mathcal Q_k}^{\top}\neq0\) implies
\((\mathcal C_{\mathcal Q})_{ik}=1\). The converse implication holds whenever
the row outer-product assembly of \(A=G^\top G\) produces no block-level
cancellations; see also \cite{krzysik2026conforming}. In that case,
\(\mathcal C_{\mathcal Q}\) is exactly the standard \(A\)-interaction graph.}

Let \(\chi_{\mathcal Q}:=\chi(\mathcal C_{\mathcal Q})\) be the chromatic number of \(\mathcal C_{\mathcal Q}\); that is,
\(\chi_{\mathcal Q}\) is the smallest number of colors needed to color the
Schwarz subdomains so that \(c(i)\neq c(k)\) whenever
\((\mathcal C_{\mathcal Q})_{ik}=1\). This coloring is admissible for the
standard additive Schwarz projection argument; see, e.g.,
\cite[Sec.~2]{toselli2004domain}. Indeed, if two indices have the same color,
then \((\mathcal C_{\mathcal Q})_{ik}=0\), so
\(R_{\mathcal Q_i}AR_{\mathcal Q_k}^{\top}=0\), and the corresponding
Schwarz subspaces are \(A\)-orthogonal. Writing
\(M_{\mathcal Q{\rm AS}}^{-1}A=\sum_{i=1}^{n_{\mathcal Q}}P_i\), with
\(P_i=R_{\mathcal Q_i}^{\top}A_{\mathcal Q_i}^{-1}R_{\mathcal Q_i}A\) and
\(A_{\mathcal Q_i}=R_{\mathcal Q_i}AR_{\mathcal Q_i}^{\top}\), each \(P_i\)
is the \(A\)-orthogonal projection onto \(\range(R_{\mathcal Q_i}^{\top})\).
Therefore, within each color class the corresponding projections sum to an
\(A\)-orthogonal projection, and summing over the \(\chi_{\mathcal Q}\)
color classes gives
\[
    \lambda_{\max}(M_{\mathcal Q{\rm AS}}^{-1}A)\le \chi_{\mathcal Q}.
\]

The following lemma records the resulting comparison between the row-touching
constant from \cref{thm:nu_touching_AS} and the coloring constant.

\begin{lemma}
\label{lem:nu_le_chi}
With \(\nu_{\mathcal Q}\) and \(\chi_{\mathcal Q}\) defined above,
\[
    \nu_{\mathcal Q}\le \chi_{\mathcal Q}.
\]
\end{lemma}

\begin{proof}
For a fixed row index \(j \in \{1,\ldots,m\}\), define
\(S_j:=\{i:(\mathcal E_{\mathcal Q})_{ji}=1\}\). If \(i,k\in S_j\) and
\(i\neq k\), then row \(j\) touches both \(\mathcal Q_i\) and
\(\mathcal Q_k\), so \((\mathcal C_{\mathcal Q})_{ik}=1\). Hence all
indices in \(S_j\) must receive distinct colors in any valid coloring of
\(\mathcal C_{\mathcal Q}\). Therefore \(\chi_{\mathcal Q}\ge |S_j|\) for
every \(j\). Taking the maximum over \(j\) gives
\(
    \chi_{\mathcal Q}
    \ge
    \max_{1\le j\le m}|S_j|
    =
    \nu_{\mathcal Q}.
\)
\end{proof}

The comparison in \cref{lem:nu_le_chi} is not generally sharp. The following
example shows that the row-touching constant can be smaller than the
corresponding coloring constant.

\begin{lemma}
\label{lem:nu_strictly_smaller_than_coloring}
The inequality \(\nu_{\mathcal Q}\le \chi_{\mathcal Q}\) can be strict.
\end{lemma}

\begin{proof}
We prove this by example, showing a class of problems for which \(
    \nu_{\mathcal Q}=n-1<n=\chi_{\mathcal Q}.
\)
Let \(n\ge 3\), set \(m=n_{\mathcal Q}=n\), and let
\(\mathcal Q_i=\{i\}\), \(i=1,\ldots,n\). Define
\[
    G
    =
    \bm 1\bm 1^\top - I
    =
    \begin{bmatrix}
    0 & 1 & 1 & \cdots & 1 \\
    1 & 0 & 1 & \cdots & 1 \\
    1 & 1 & 0 & \ddots & \vdots \\
    \vdots & \vdots & \ddots & \ddots & 1 \\
    1 & 1 & \cdots & 1 & 0
    \end{bmatrix}.
\]
The eigenvalues of \(G\) are \(n-1\) and \(-1\) with multiplicity \(n-1\), so
\(G\) has full column rank. 
Hence, all assumptions of \cref{thm:nu_touching_AS} are met.
Since the Schwarz subdomains are singletons and
\(G\) has only zero-one entries, \(\mathcal E_{\mathcal Q}=G\). Hence each
row of \(\mathcal E_{\mathcal Q}\) has \(n-1\) nonzero entries, and
\(\nu_{\mathcal Q} = \| {\cal E}_{\cal Q} \|_{\infty} = n-1\).

Since \(n\ge 3\), every pair of distinct columns of
\(\mathcal E_{\mathcal Q}\) has a common nonzero entry. Therefore
\(
    \mathcal C_{\mathcal Q}
    =
    \bm 1\bm 1^\top-I
    =
    \mathcal E_{\mathcal Q}.
\)
Thus \(\chi_{\mathcal Q}=n\), and consequently
\(
    \nu_{\mathcal Q}=n-1<n=\chi_{\mathcal Q}.
\)
\end{proof}

The preceding example is purely algebraic; it shows that the comparison
\(\nu_{\mathcal Q}\le \chi_{\mathcal Q}\) need not be sharp under the
abstract hypotheses of \cref{thm:nu_touching_AS}. In
\cref{sec:AS_num}, we examine the corresponding quantities for the
coordinate covers arising from our aggregate construction. There we compare
\(\nu_\Omega\) with \(\chi_\Omega\) and with practical upper bounds for
\(\chi_\Omega\), in order to assess how much is lost by replacing the
row-touching estimate with a coloring-based one.

\subsection{Block-Jacobi for $\mathsf{M}$-matrices}
\label{sec:SPDDM}

Upon specializing \cref{cor:damped_two_level_bound} to block Jacobi,
\(M=M_{\omega{\rm J}}\), the only remaining smoother-dependent quantity is
\(\lambda_{\max}(M_{\omega{\rm J}}^{-1}A)\). The general Gram-touching
argument in \cref{thm:nu_touching_AS} bounds this quantity in
terms of multiplicity-like constants. We next record a sharper bound for the
case in which \(A\) is an $\mathsf{M}$-matrix.
Under the standing assumption \(A=A^\top\succ0\), the additional condition
\(a_{pq}\le0\) for \(p\neq q\) in the lemma below means that \(A\) is a Stieltjes matrix
\cite[Def. 3.23]{varga1999matrix}, equivalently a symmetric nonsingular $\mathsf{M}$-matrix
\cite[Cor. 3.24]{varga1999matrix}. This class includes many standard scalar finite-difference discretizations of Laplacian-like operators.
Remarkably, this result can then be used to show uniform two-level convergence bounds of the form $K_{\rm TG} \lesssim 2 \tau_{\max}$. %

\begin{lemma}
\label{lem:A_le_2_MomegaJ}
Let \(A=(a_{pq})\) satisfy \(a_{pq}\le0\) for \(p\neq q\), and let
\(M_{\omega{\rm J}}\) be as in \eqref{eq:Mjac_def}. Then
\[
    \lambda_{\max}(M_{\omega{\rm J}}^{-1}A)<2 .
\]
\end{lemma}

\begin{proof}
Let us write $A = M_{\omega{\rm J}} - \Delta$, such that $\Delta := M_{\omega{\rm J}} - A$.
Since the matrix \(M_{\omega{\rm J}}\) is obtained from \(A\) by setting some of the off-diagonal entries in \(A\) to zero and since \(A\) is a nonsingular \(\mathsf{M}\)-matrix, the regular-splitting result
\cite[Thm. 3.31]{varga1999matrix} gives
\(
    \rho(M_{\omega{\rm J}}^{-1}\Delta)<1 .
\)
Using \(\Delta=M_{\omega{\rm J}}-A\), this is
\(
    \rho(I-M_{\omega{\rm J}}^{-1}A)<1 .
\)
Moreover, \(M_{\omega{\rm J}}\) is SPD because it is block diagonal with SPD principal subblocks of \(A\). Hence \(M_{\omega{\rm J}}^{-1}A\) is similar to
\(M_{\omega{\rm J}}^{-1/2} A M_{\omega{\rm J}}^{-1/2}\), and therefore has real positive eigenvalues. If
\(\mu\) is an eigenvalue of \(M_{\omega{\rm J}}^{-1}A\), then \(1-\mu\) is an eigenvalue of
\(I-M_{\omega{\rm J}}^{-1}A\), so \(|1-\mu|<1\). Thus \(0<\mu<2\), and consequently
\(
    \lambda_{\max}(M_{\omega{\rm J}}^{-1}A)
    <2 .
\)
\end{proof}

Substituting this result into \cref{cor:damped_two_level_bound} gives the following.  

\begin{corollary}
\label{cor:ddm_block_jacobi_two_level}
Assume the hypotheses of \cref{lem:A_le_2_MomegaJ} and let 
\(M = \zeta^{-1}M_{\omega{\rm J}}\) with \(0<\zeta<1\).
Then
\[
K_{\rm TG}^{\zeta}
\le
\frac{\tau_{\max}}{2\zeta(1-\zeta)} .
\]
In particular, 
\(
K_{\rm TG}^{\zeta_*}\le 2\tau_{\max} 
\)
for \(\zeta_*=1/2\).
\end{corollary}

\section{Comparison with convergence theory based on the fictitious subspace lemma}
\label{sec:theory_interpretation}

Developing a fair comparison between the new theory developed in this paper and existing DD theory is rather nuanced, depending on many factors, including whether one is comparing fixed-point and preconditioned iteration, two-level and multilevel methods, single-step and asymptotic bounds, or identical versus related operators. 
The purpose of this section is not to make such a rigorous comparison but merely to provide a qualitative comparison with related DD-type methods and the associated convergence bounds that have been developed with alternative techniques. 

Two-level convergence theory for an algorithm closely related to LS-AMG-DD was previously developed in~\cite{aldaas2022normaleq}. The present analysis differs from that work in two main respects. First,~\cite{aldaas2022normaleq} studies an \emph{additive} two-level preconditioner in combination with an additive Schwarz smoother, whereas in this paper we analyze a \emph{multiplicative} two-grid realization in the form of a $V(1,1)$ cycle. Second, the theoretical framework is different: the analysis in~\cite{aldaas2022normaleq}, and more broadly in spectral overlapping domain decomposition methods, is typically based on the fictitious subspace lemma, while we derive bounds using multigrid-style approximation properties.

\paragraph{Additive two-level Schwarz and fictitious-subspace analysis}
A commonly analyzed method in spectral domain decomposition is the additive two-level Schwarz preconditioner
\[
B_{\mathrm{add}}^{-1} = P A_c^{-1} P^\top + M^{-1}_{\Omega {\rm AS}},
\]
which yields the error-propagation operator
\[
E_{\mathrm{add}} = I-\Pi_A - M^{-1}_{\Omega {\rm AS}} A.
\]
Using the new bound $ \lambda_{\max}(M_{\Omega {\rm AS}}^{-1}A) \leq \nu_\Omega$ from \Cref{thm:nu_touching_AS} and assuming an exact local SPSD splitting (implying $k_m=1$ in the context of \cite{aldaas2022normaleq}), yields the convergence bound
\[
\kappa(B_{\mathrm{add}}^{-1}A) \;\le\; (\nu_\Omega+1)\Bigl(2+(2\nu_\Omega+1)\tau_{\max}\Bigr),
\]
where $\tau_{\max}$ is defined as in~\eqref{eq:tau_i_def}. The additive formulation is attractive from an analytical standpoint—particularly within the fictitious subspace lemma—but it is less commonly used in practice due to its inferior performance compared with multiplicative variants.

\paragraph{Balanced Neumann--Neumann (BNN) correction}
A standard multiplicative alternative is the so-called balanced Neumann--Neumann (BNN) correction,
\[
B_{\mathrm{BNN}}^{-1} = P A_c^{-1}P^\top + (I-\Pi_A)\,M^{-1}_{\Omega {\rm AS}}\,(I-\Pi_A^\top),
\]
which typically improves practical convergence. It is important, however, to distinguish algorithmic cost: applying $B_{\mathrm{BNN}}^{-1}$ to a vector involves three coarse-space corrections, whereas a $V(1,1)$ cycle requires only a single coarse correction. Using the identity $\Pi_A^\top A = A\Pi_A$, the action of $B_{\mathrm{BNN}}^{-1}A$ can be implemented with two coarse corrections, but it remains more expensive than the $V(1,1)$ cycle in this respect.

The error-propagation operator associated with $B_{\mathrm{BNN}}^{-1}$ is
\[
E_{\mathrm{BNN}} = (I-\Pi_A)\bigl(I-M^{-1}_{\Omega {\rm AS}} A+ M^{-1}_{\Omega {\rm AS}}A\Pi_A\bigr).
\]
Using the fictitious subspace lemma, the new bound $ \lambda_{\max}(M_{\Omega {\rm AS}}^{-1}A) \leq \nu_\Omega$ from \Cref{thm:nu_touching_AS}, and assuming an exact local SPSD splitting, one obtains the bound
\[
\kappa(B_{\mathrm{BNN}}^{-1}A) \;\le\; \nu_\Omega\,(1+\tau_{\max}),
\]
see, e.g.,~\cite[Chapter~7]{DolJN15}. 

\paragraph{Relationship to the $V(1,1)$ cycle}
Although the error-propagation operators for the $V(1,1)$ cycle analyzed here and for the BNN preconditioner are not identical, asymptotically they are closely related, in the sense that their powers exhibit the same dominant factors. Specifically, raising the two-grid error-propagation operator $E_{\rm TG}$ from \eqref{eq:ETG} and $E_{\mathrm{BNN}}$ to the power $k$ yields
\[
E_{\rm TG}^k
= (I-M_{\Omega {\rm AS}}^{-\top}A)\left((I-\Pi_A)(I-\widetilde{M}_{\Omega {\rm AS}}^{-1}A)\right)^{k-1}(I-\Pi_A)(I-M_{\Omega {\rm AS}}^{-1}A),
\]
and
\[
E_{\mathrm{BNN}}^k
= \left((I-\Pi_A)(I-M_{\Omega {\rm AS}}^{-1}A)\right)^{k-1}(I-\Pi_A)\bigl(I-M_{\Omega {\rm AS}}^{-1}A+M_{\Omega {\rm AS}}^{-1}A\Pi_A\bigr).
\]
These expressions make clear that both schemes repeatedly apply the same core contraction mechanism driven by $(I-\Pi_A)$ and the Schwarz smoothing step.

The bound in \Cref{lem:two_level_concrete} of \(K_{\mathrm{TG}}^{\zeta_*} \le \nu_\Omega \tau_{\max}\) is smaller than the BNN bound quoted above of \(\kappa(B_{\mathrm{BNN}}^{-1}A) \;\le\; \nu_\Omega (1+\tau_{\max})\). 
However, this comparison is not between identical operators.  The BNN estimate is a condition-number bound for a two-level preconditioner, typically used with a Krylov method.  The estimate in \Cref{lem:two_level_concrete} is instead a contraction bound for a stationary two-level cycle using a damped smoother.  
Thus the comparison only shows that, for this cycle, the multigrid argument gives a smaller constant than the quoted BNN bound, but that does not imply that the \(V(1,1)\) cycle is a better preconditioner than the balanced DD method.

\section{Numerical experiments}
\label{sec:numerics}

We now numerically test the theory using moderate-sized model problems for which the coarse-space, smoother, and observed two-level quantities can be measured directly. 

\subsection{Two-level numerical setup}
\label{subsec:numerics_two_level_setup}

For each model problem described next in \cref{subsec:numerics_model_problems}, we construct a
single two-level LS-AMG-DD hierarchy implemented in PyAMG \cite{bell2023pyamg}.  
Aggregates are built by two passes of
standard aggregation on the graph of \(A\), with no strength-of-connection
filter. 
The interpolation $P$ is constructed according to \cref{def:coarse_space}, and the Galerkin coarse operator
\(P^{\top}AP\) is inverted directly.
For each smoother \(M\) we consider (detailed below), we report an observed two-level constant \(K_M^{\rm obs} \approx K_{\rm TG}\) measured by applying the two-level method to the homogeneous system \(A\bm x=\bm 0\).  
Since \(\bm x = \bm 0\), the iterate \(\bm{x}_k\) is also the algebraic error, $\bm{e}_k := \bm{x}_k - \bm{0} = \bm{x}_k$.
Each reported \(K_M^{\rm obs}\) is the worst-case constant from applying 100 two-level iterations across 10 different randomly chosen initial vectors $\{\bm{x}_0^s \}_{s = 1}^{10}$.
That is, we define
\[
  \rho_M^{\rm obs}
  :=
  \max_{\substack{1\le s\le 10\\0\le k<100}}
  \frac{\|\bm e_{k+1}^s\|_A}{\|\bm e_{k}^s\|_A},
  \qquad
  K_M^{\rm obs}
  :=
  \frac{1}{1-\rho_M^{\rm obs}},
\]
with \(K_M^{\rm obs}=\infty\) if \(\rho_M^{\rm obs}\ge 1\). 

In our diagnostics we also report $\tau_{\rm cut}$ as well as the realized cutoff \(\tau_{\max} \leq \tau_{\rm cut}\). 
We consider the damped smoother $M_{\zeta} = \zeta^{-1} M$ for $M\in\{M_{\omega{\rm J}},M_{\Omega{\rm AS}}\}$ with damping \(\zeta = \zeta_*=\lambda_{\max}(M^{-1}A)^{-1}\). 
For this we report numerically observed values $K_{M}^{\rm obs}$, as well as the bound \( \texttt{bound}_{M}
  :=
  \lambda_{\max}(M^{-1}A) \tau_{\max} \geq K_{\rm TG}^{\zeta_*}
\) coming from combining \cref{cor:damped_two_level_bound} and \cref{lem:as_vs_MaggJac}.
The spectral radii
\(\lambda_{\max}(M_{\omega{\rm J}}^{-1}A)\) and
\(\lambda_{\max}(M_{\Omega{\rm AS}}^{-1}A)\) are computed by PyAMG's \texttt{approximate\_spectral\_radius}.  
We also report $K_{M}^{\rm obs}$ for the nonsymmetric smoothers $ M \in \{M_{\Omega {\rm RAS}}, M_{\Omega {\rm MS}}\}$, although we do not explicitly compute the (conditional) bound of \(
\lambda_{\max}(M^{-1}_{\omega{\rm J}} \widetilde M)\,\tau_{\max}
\) from \cref{thm:two_level_general}.
The RAS method is as described in \cite{southworth2026lsamgdd}, and the multiplicative Schwarz algorithm is the standard, naturally-ordered multiplicative Schwarz algorithm defined on the overlaps \eqref{eq:overlap_def}; see also the discussion in \cref{sec:nonsymmetric_smoother_case}.

Finally, we report the minimal Jacobi WAP constant ${\mathcal W_{M_{\omega{\rm J}}}}$. From \cref{lem:Q0_basic}, recall the $M_{\omega {\rm J}}$-orthogonal projection onto $\range(P)$ denoted by \(\Pi_{M_{\omega {\rm J}}}\). 
Then from \cref{def:wap_calM} we have
\begin{align*}
    {\mathcal W_{M_{\omega{\rm J}}}} 
    = 
    \sup_{\bm v \neq 0}
    \frac{\| (I - \Pi_{M_{\omega {\rm J}}}) \bm{v} \|^2_{M_{\omega {\rm J}}}}{\| \bm{v} \|_A^2}
    =
    \sup_{\bm v \neq 0}
    \frac{
    \bm{v}^\top \big[ (I - \Pi_{M_{\omega {\rm J}}})^\top
    M_{\omega {\rm J}}
    (I - \Pi_{M_{\omega {\rm J}}}) \big] \bm{v}
    }
    {
    \bm{v}^\top A \bm{v}
    }.
\end{align*}
We numerically evaluate this generalized eigenvalue with \texttt{sparse.linalg.eigsh} from SciPy  \cite{virtanen2010scipy}, using a tolerance of \(10^{-6}\); 
applications of \(A^{-1}\) are applied with a direct solver and the action of \(I - \Pi_{M_{\omega {\rm J}}}\) is computed matrix free using the expression in \cref{lem:Q0_basic}.

\subsection{Model problems}
\label{subsec:numerics_model_problems}

We consider three conforming finite element model operators.  The scalar
\(H^1({\cal D})\) and vector \(H(\operatorname{div}; {\cal D})\) problems are posed on the two-dimensional
spatial domain \({\cal D} = (0,1)^2\), discretized on triangular meshes, while the
vector \(H(\operatorname{curl}; {\cal D})\) problems are posed on the three-dimensional spatial
domain \({\cal D} = (0,1)^3\) and discretized on tetrahedral meshes. 
We write
\(\mathrm{CG}_p\) for the continuous Lagrange space of degree \(p\),
\(\mathrm{RT}_p\) for the Raviart--Thomas space of degree \(p\),
and
\(\mathrm{N1}_p\) for the first-family N\'ed\'elec edge space of degree \(p\).
The \(H^1\) and \(H(\operatorname{div})\) tests use \(p\in\{1,3\}\), while
the \(H(\operatorname{curl})\) tests use \(p\in\{1,2\}\).
All problems are implemented with the Firedrake finite-element package \cite{FiredrakeUserManual}.
The matrices $A = G^\top G$ used in the diagnostics are assembled from one of the three bilinear forms
\begin{subequations}
\begin{align*}
  a_1(u_h,v_h)
  &:=
  (K\operatorname{grad} u_h,\operatorname{grad} v_h)_{{\cal D}}
  + b_{1}(u_h,v_h),
  && u_h,v_h\in\mathrm{CG}_p,
  \\
  a_{\operatorname{div}}(\bm q_h,\bm r_h)
  &:=
  (\alpha \operatorname{div}\bm q_h,\operatorname{div}\bm r_h)_{{\cal D}}
  +(\bm q_h,\bm r_h)_{{\cal D}}
  +b_{\operatorname{div}}(\bm q_h,\bm r_h),
  && \bm q_h,\bm r_h\in\mathrm{RT}_p,
  \\
  a_{\operatorname{curl}}(\bm w_h,\bm z_h)
  &:=
  (\alpha \operatorname{curl}\bm w_h,\operatorname{curl}\bm z_h)_{{\cal D}}
  +(\bm w_h,\bm z_h)_{{\cal D}}
  +b_{\operatorname{curl}}(\bm w_h,\bm z_h),
  && \bm w_h,\bm z_h\in\mathrm{N1}_p,
\end{align*}
\end{subequations}
using broken finite-element spaces, as detailed in \cite{krzysik2026conforming}.
For the scalar diffusion problem, the diffusion tensor is
\[
K = Q_\theta D_\epsilon Q_\theta^\top,
\qquad
  Q_\theta =
  \begin{pmatrix}
    \cos\theta & -\sin\theta \\
    \sin\theta & \cos\theta
  \end{pmatrix},
  \qquad
  D_\epsilon =
  \begin{pmatrix}
    1 & 0 \\
    0 & \epsilon
  \end{pmatrix},
\]
with \(\epsilon=10^{-3}\) and \(\theta=\pi/6\), corresponding to non-mesh-aligned anisotropy.  
For the vector problems, we use the coefficient \(\alpha=10^3\). 

Homogeneous essential boundary conditions are imposed weakly on all boundaries using coefficient-scaled symmetric trace penalties. Note that, as described in \cite{krzysik2026conforming}, boundary conditions are imposed weakly for ease of numerical implementation and not because of any limitations of the solver framework. 
Specifically, with
\(\gamma_p=36p^2\), the boundary forms used above are defined by
\begin{subequations}
\label{eq:numerics_boundary_penalties}
\begin{align*}
  b_{1}(u,v)
  &:=
  \sum_{F\subset\partial{\cal D}}
  \left\langle
    \frac{\gamma_p\,\bm n^{\top}K\bm n}{h_F}u,v
  \right\rangle_F,
  \\
  b_{\operatorname{div}}(\bm q,\bm r)
  &:=
  \sum_{F\subset\partial{\cal D}}
  \left\langle
    \gamma_p\left(\frac{\alpha}{h_F}+h_F\right)
    (\bm q\cdot\bm n)(\bm r\cdot\bm n)
  \right\rangle_F,
  \\
  b_{\operatorname{curl}}(\bm w,\bm z)
  &:=
  \sum_{F\subset\partial{\cal D}}
  \left\langle
    \gamma_p\left(\frac{\alpha}{h_F}+h_F\right)
    (\bm n\times\bm w)\cdot(\bm n\times\bm z)
  \right\rangle_F .
\end{align*}
\end{subequations}
Here \(\bm n\) is the unit outward normal and \(h_F\) is the local facet
diameter.  Thus the weakly imposed essential traces are \(u_h=0\),
\(\bm q_h\cdot\bm n=0\), and \(\bm n\times\bm w_h=\bm 0\), respectively.

\subsection{Results}
\label{subsec:numerics_results}

The data are shown in
\cref{tab:h1_two_level_diagnostics,tab:hdiv_two_level_diagnostics,tab:hcurl_two_level_diagnostics}.
Every table entry has the form
``\(
  \text{value on a coarse mesh} \,/\, \text{value on a once-refined mesh},
\)''. The same PDE problem uses different mesh resolutions for different \(p\).
For each problem and polynomial degree, the caption reports the corresponding
total number of algebraic DOFs \(n_h\) on the coarse and once-refined meshes.  All values are
rounded to one decimal place.

\begin{table}[h!]
\caption{Two-level diagnostics for the two-dimensional, scalar \(H^1\) rotated anisotropic
diffusion test with \(\epsilon=10^{-3}\), \(\theta=\pi/6\).  Entries are reported as
coarse mesh/refined mesh.  For \(p=1\), the meshes have \(n_h=16641/66049\)
algebraic DOFs.  For \(p=3\), the meshes have \(n_h=9409/37249\)
algebraic DOFs.
\label{tab:h1_two_level_diagnostics}
}
\centering
\setlength{\tabcolsep}{2.5pt}
\begin{tabular}{c|cc|cc|cc|cc}
\hline
\(\tau_{\rm cut}\)
&
\(\tau_{\max}\)
&
\({\mathcal W_{M_{\omega{\rm J}}}}\)
&
\(\texttt{bound}_{\omega{\rm J}}\)
&
\(K_{\omega{\rm J}}^{\rm obs}\)
&
\(\texttt{bound}_{\Omega{\rm AS}}\)
&
\(K_{\Omega{\rm AS}}^{\rm obs}\)
&
\(K_{\Omega {\rm RAS}}^{\rm obs}\)
&
\(K_{\Omega {\rm MS}}^{\rm obs}\)
\\
\hline
\multicolumn{9}{c}{\(p=1\)} \\
\hline
\(2\)  & \(2.0/2.0\) & \(1.8/1.8\) & \(4.0/4.0\)   & \(2.1/2.1\) & \(6.0/7.4\)   & \(1.8/2.2\) & \(1.1/1.1\)        & \(1.1/1.1\) \\
\(5\)  & \(5.0/5.0\) & \(4.1/4.1\) & \(10.0/10.0\) & \(4.9/4.7\) & \(14.9/18.4\) & \(3.7/4.0\) & \(4.5/4.2\)        & \(2.1/2.0\) \\
\(10\) & \(9.4/9.4\) & \(6.1/6.1\) & \(19.0/19.1\) & \(8.2/8.1\) & \(28.2/35.0\) & \(7.8/8.8\) & \(25.9/\infty\)   & \(6.3/6.6\) \\
\hline
\multicolumn{9}{c}{\(p=3\)} \\
\hline
\(2\)  & \(2.0/2.0\) & \(1.6/1.7\) & \(4.1/4.1\)   & \(2.1/2.3\) & \(8.0/8.0\)   & \(2.3/2.3\) & \(1.1/1.1\)        & \(1.3/1.3\) \\
\(5\)  & \(5.0/5.0\) & \(3.4/3.4\) & \(10.2/10.2\) & \(4.3/3.8\) & \(19.9/19.9\) & \(2.9/2.7\) & \(1.5/1.4\)        & \(1.4/1.4\) \\
\(10\) & \(9.9/9.9\) & \(6.4/6.6\) & \(20.3/20.4\) & \(7.9/7.5\) & \(39.5/39.8\) & \(4.6/3.8\) & \(\infty/\infty\) & \(2.0/2.0\) \\
\hline
\end{tabular}
\end{table}

\begin{table}[h!]
\caption{Two-level diagnostics for the two-dimensional, vector \(H(\operatorname{div})\) model problem discretized with \(\mathrm{RT}_p\) elements and \(\alpha=10^3\).  Entries are reported as coarse mesh/refined mesh.  For \(p=1\), the meshes have
\(n_h=12416/49408\) algebraic DOFs.  For \(p=3\), the meshes have
\(n_h=5472/21696\) algebraic DOFs.
\label{tab:hdiv_two_level_diagnostics}
}
\centering
\setlength{\tabcolsep}{2.5pt}
\begin{tabular}{c|cc|cc|cc|cc}
\hline
\(\tau_{\rm cut}\)
&
\(\tau_{\max}\)
&
\({\mathcal W_{M_{\omega{\rm J}}}}\)
&
\(\texttt{bound}_{\omega{\rm J}}\)
&
\(K_{\omega{\rm J}}^{\rm obs}\)
&
\(\texttt{bound}_{\Omega{\rm AS}}\)
&
\(K_{\Omega{\rm AS}}^{\rm obs}\)
&
\(K_{\Omega {\rm RAS}}^{\rm obs}\)
&
\(K_{\Omega {\rm MS}}^{\rm obs}\)
\\
\hline
\multicolumn{9}{c}{\(p=1\)} \\
\hline
\(2\)  & \(2.0/2.0\) & \(1.0/1.1\) & \(6.0/6.0\)   & \(1.8/1.8\) & \(7.9/7.9\)   & \(2.3/2.3\) & \(1.0/1.0\) & \(1.1/1.1\) \\
\(5\)  & \(3.9/3.9\) & \(3.9/3.9\) & \(11.7/11.7\) & \(7.2/7.2\) & \(15.5/15.5\) & \(5.9/5.6\) & \(8.3/8.4\) & \(3.9/4.0\) \\
\(10\) & \(3.9/3.9\) & \(3.9/3.9\) & \(11.7/11.7\) & \(7.2/7.2\) & \(15.5/15.5\) & \(5.9/5.6\) & \(8.3/8.4\) & \(3.9/4.0\) \\
\hline
\multicolumn{9}{c}{\(p=3\)} \\
\hline
\(2\)  & \(2.0/2.0\) & \(1.2/1.2\) & \(6.0/6.0\)   & \(1.9/1.9\)  & \(6.0/8.0\)   & \(1.8/2.3\) & \(1.2/1.2\) & \(1.4/1.4\) \\
\(5\)  & \(3.5/3.5\) & \(1.3/1.3\) & \(10.4/10.4\) & \(2.2/2.2\)  & \(10.4/13.8\) & \(1.8/2.3\) & \(1.3/1.3\) & \(1.6/1.5\) \\
\(10\) & \(9.5/9.4\) & \(9.4/8.2\) & \(28.5/28.3\) & \(17.3/15.4\) & \(28.5/37.7\) & \(6.7/7.5\) & \(4.5/4.6\) & \(2.7/2.6\) \\
\hline
\end{tabular}
\end{table}

\begin{table}[h!]
\caption{Two-level diagnostics for the three-dimensional, vector
\(H(\operatorname{curl})\) model problem discretized with \(\mathrm{N1}_p\)
elements and \(\alpha=10^3\).  Entries are reported as coarse/refined.  For
\(p=1\), the meshes have \(n_h=4184/31024\) algebraic DOFs.  For \(p=2\),
the meshes have \(n_h=2936/21424\) algebraic DOFs.
\label{tab:hcurl_two_level_diagnostics}
}
\centering
\setlength{\tabcolsep}{2.5pt}
\begin{tabular}{c|cc|cc|cc|cc}
\hline
\(\tau_{\rm cut}\)
&
\(\tau_{\max}\)
&
\({\mathcal W_{M_{\omega{\rm J}}}}\)
&
\(\texttt{bound}_{\omega{\rm J}}\)
&
\(K_{\omega{\rm J}}^{\rm obs}\)
&
\(\texttt{bound}_{\Omega{\rm AS}}\)
&
\(K_{\Omega{\rm AS}}^{\rm obs}\)
&
\(K_{\Omega {\rm RAS}}^{\rm obs}\)
&
\(K_{\Omega {\rm MS}}^{\rm obs}\)
\\
\hline
\multicolumn{9}{c}{\(p=1\)} \\
\hline
\(2\)  & \(2.0/2.0\) & \(1.5/1.6\) & \(5.4/5.6\)   & \(2.3/2.4\)  & \(14.0/14.1\) & \(3.8/3.8\) & \(1.1/1.1\)       & \(1.1/1.1\) \\
\(5\)  & \(4.9/5.0\) & \(3.2/3.5\) & \(13.3/14.1\) & \(4.8/5.1\)  & \(34.4/35.1\) & \(3.9/4.0\) & \(1.9/2.5\)       & \(1.2/1.6\) \\
\(10\) & \(9.5/9.9\) & \(5.0/6.2\) & \(25.6/27.8\) & \(7.5/10.0\) & \(66.3/69.3\) & \(5.8/8.5\) & \(3.1/7.0\)       & \(1.4/2.4\) \\
\hline
\multicolumn{9}{c}{\(p=2\)} \\
\hline
\(2\)  & \(2.0/2.0\)  & \(1.5/1.5\) & \(5.8/6.4\)   & \(2.4/2.6\)  & \(10.3/14.0\) & \(2.9/3.8\) & \(1.1/1.1\)      & \(1.1/1.1\) \\
\(5\)  & \(5.0/5.0\)  & \(3.3/3.7\) & \(14.5/15.9\) & \(5.2/6.0\)  & \(25.9/35.0\) & \(2.9/3.8\) & \(1.4/1.5\)      & \(1.2/1.2\) \\
\(10\) & \(9.8/10.0\) & \(6.4/6.5\) & \(28.5/31.8\) & \(9.8/10.6\) & \(50.9/70.0\) & \(3.1/4.8\) & \(11.8/\infty\) & \(1.2/1.5\) \\
\hline
\end{tabular}
\end{table}

The displayed data support the main two-level estimates in
\cref{sec:two_level} and \cref{sec:smoothers}.  
The first comparison is between the realized
cutoff \(\tau_{\max}\) and the minimal WAP constant
\({\mathcal W_{M_{\omega{\rm J}}}}\).  In all reported cases, \({\mathcal W_{M_{\omega{\rm J}}}}\le \tau_{\max}\), as
predicted by \cref{thm:jacobi_WAP}.
The comparison is also fairly
sharp in many rows, especially at larger \(\tau_{\rm cut}\).  
We note
that \(\tau_{\max}\) is a post-setup quantity determined by the local spectra:
it need not coincide with the prescribed threshold \(\tau_{\rm cut}\).  This is visible in
several rows where the first discarded finite eigenvalue is slightly smaller than the prescribed input cutoff $\tau_{\rm cut}$, and particularly in the $p=1$ rows of the \(H(\operatorname{div})\) table, wherein the realized coarse space is the same for both $\tau_{\rm cut} = 5$ and $\tau_{\rm cut} = 10$.

The symmetric two-level constants show the same basic behavior.  The observed
constants \(K_{\omega{\rm J}}^{\rm obs}\) and
\(K_{\Omega{\rm AS}}^{\rm obs}\) are bounded by their corresponding quantities \(\texttt{bound}_{\omega{\rm J}}\) and \(\texttt{bound}_{\Omega{\rm AS}}\) throughout the
tables. These bounds track the correct scale of the observed constants across all three model classes, degrees, and mesh refinements.  
The block-Jacobi bound is generally closer to the observed value than the overlapping-additive-Schwarz bound.  This is consistent with the proof structure, since the WAP is measured directly in the
aggregate block-Jacobi norm, while the additive-Schwarz estimate passes through an additional comparison of smoother metrics. We note that in cases of smaller $\tau_{\rm cut}$, the observed block Jacobi constants are typically smaller than the overlapping additive-Schwarz two-level constants, but that this trend tends to reverse for larger $\tau_{\rm cut}$. 

The RAS and overlapping multiplicative Schwarz columns generally appear consistent with the general bound \(K_{\mathrm{TG}}
\le 
\lambda_{\max}(M^{-1}_{\omega{\rm J}} \widetilde M)\,\tau_{\max}\) from \cref{thm:two_level_general}, in that the observed two-level constants increase with $\tau_{\rm cut}$. 
However, it is difficult to say more, since we do not currently have a way of (practically) estimating or bounding \(\lambda_{\max}(M^{-1}_{\omega{\rm J}} \widetilde M)\).
It is interesting that the overlapping multiplicative Schwarz two-level cycle is by far the most robust of the observed nonsymmetric
variants in these experiments, with \(K_{\Omega {\rm MS}}^{\rm obs}\) remaining comparatively small and showing only mild growth as \(\tau_{\rm cut}\) is increased.  
The RAS behavior is
less uniform.
In fact, the \(H^1\) and \(H(\operatorname{curl})\) tables contain infinite values of \(K_{\Omega {\rm RAS}}^{\rm obs}\) for \(\tau_{\rm cut}=10\), indicating that the two-level solver diverged. 
We expect that in these cases the corresponding smoother is not contractive in the $A$-norm, such that the bound \(K_{\mathrm{TG}}
\le 
\lambda_{\max}(M^{-1}_{\omega{\rm J}} \widetilde M)\,\tau_{\max}\) does not hold because the smoother does not meet the hypothesis required by \cref{thm:two_level_general}.\footnote{For some coarser-mesh $H^1$ problems than reported in \cref{tab:h1_two_level_diagnostics}, we have verified by dense computation that the operator \(M_{\Omega{\rm RAS}}+M_{\Omega{\rm RAS}}^\top-A\) is indefinite, corresponding to the smoother not being contractive in the \(A\)-norm.} 
This is consistent with the discussion in \cref{sec:nonsymmetric_smoother_case}, which showed that this smoother is not guaranteed to be convergent.
We also highlight that, despite the fact that RAS is not convergent in some cases, in many cases it yields a faster converging two-level method than either block Jacobi or overlapping additive Schwarz.

There is little evidence of systematic convergence deterioration under the displayed mesh
refinements or increases in polynomial degree.  The values
change from case to case, and larger thresholds $\tau_{\rm cut}$ naturally lead to larger
constants because fewer local eigenvectors are retained, but the constants do
not show uncontrolled growth with either \(h^{-1}\) or \(p\).  This is also consistent with the theory: after the local spectral threshold $\tau_{\rm cut}$ has been
fixed, the two-level constant is controlled by the scale of the smoother iteration, which should not depend on $h$ or $p$, as per \cref{sec:smoothers}.

\subsection{Overlapping additive Schwarz}
\label{sec:AS_num}

Finally, we consider the additive Schwarz theory discussed in \cref{sec:AS} by specializing the generic cover $\{ {\cal Q}_i \}$ introduced in \cref{thm:nu_touching_AS} to the overlapping additive Schwarz case,
\(\mathcal Q_i=\Omega_i\).
These diagnostics use the same test setups as described in the preceding numerical sections, but
do not construct the LS-AMG-DD coarse space and do not run a two-level method. 
The goal is only to compare the spectral radius
\(\lambda_{\max}^{\rm num}(M_{\Omega{\rm AS}}^{-1}A)\approx\lambda_{\max}(M_{\Omega{\rm AS}}^{-1}A)\)
with several computable structural upper bounds, where \(\lambda_{\max}^{\rm num}(M_{\Omega{\rm AS}}^{-1}A)\) is computed with PyAMG's \texttt{approximate\_spectral\_radius}.

Recall that \(A=G^\top G=\sum_{j=1}^m \bm g_j\bm g_j^\top\), where
\(\bm g_j^\top\) is the \(j\)th row of \(G\). 
For the overlapping Schwarz domains
\(\{\Omega_i\}\), we form the binary row-overlap incidence matrix
\(\mathcal E_\Omega\) described in \cref{sec:AS}, recalling that the row-touching constant from \cref{thm:nu_touching_AS} is
\(\nu_\Omega:=\max_j\sum_i(\mathcal E_\Omega)_{ji} = \| \mathcal E_\Omega \|_{\infty}\).
For comparison with standard coloring-based estimates, we also form the row-induced overlap
interaction matrix \(\mathcal C_\Omega\) described in \cref{sec:AS}. 
We denote by \(\chi_\Omega\) the exact chromatic number of this interaction matrix, and by
\(\chi_\Omega^{\rm MIS} \geq \chi_\Omega\) the upper bound that is the number of colors returned by PyAMG's \texttt{graph.vertex\_coloring} with \texttt{method="MIS"} \cite{bell2023pyamg}. 
We compute the exact chromatic number of ${\cal C}_{\Omega}$ using
\texttt{chromatic\_number} from the GCol package \cite{lewis2025gcol}.
Since exact graph coloring is NP-hard, this is not intended as a practical setup
quantity; it is included here only as a diagnostic, and is feasible because the
interaction graphs in these tests are small enough.

\begin{table}[h!]
\centering
\small
\begin{tabular}{cc|cccc}
\hline
problem & \(p\) &
\(\lambda_{\max}^{\rm num}(M_{\Omega{\rm AS}}^{-1}A)\) &
\(\nu_\Omega\) &
\(\chi_\Omega\) &
\(\chi_\Omega^{\rm MIS}\) \\
\hline
\(H^1\) & 1 & \(3.0/3.7\) & \(3/4\) & \(3/4\) & \(4/5\) \\
\(H^1\) & 3 & \(4.0/4.0\) & \(4/4\) & \(4/4\) & \(4/4\) \\
\(H(\operatorname{div})\) & 1 & \(4.0/4.0\) & \(4/4\) & \(4/4\) & \(5/5\) \\
\(H(\operatorname{div})\) & 3 & \(3.0/4.0\) & \(3/4\) & \(3/4\) & \(4/4\) \\
\(H(\operatorname{curl})\) & 1 & \(7.0/7.0\) & \(8/8\) & \(8/9\) & \(9/11\) \\
\(H(\operatorname{curl})\) & 2 & \(5.2/7.0\) & \(6/8\) & \(6/9\) & \(6/12\) \\
\hline
\end{tabular}
\caption{Overlapping additive Schwarz diagnostics for the same test setups as in the preceding numerical sections. 
The \(H^1\) and \(H(\operatorname{div})\) tests are two-dimensional, while the
\(H(\operatorname{curl})\) tests are three-dimensional.
Entries are reported as ``value on a coarse mesh/value on a once-refined mesh.''
The quantity
\(\lambda_{\max}^{\rm num}(M_{\Omega{\rm AS}}^{-1}A)\) is the PyAMG estimate of
\(\lambda_{\max}(M_{\Omega{\rm AS}}^{-1}A)\) rounded to one decimal place.}
\label{tab:oas_smoother_diagnostics}
\end{table}

The data in \cref{tab:oas_smoother_diagnostics} show that the row-touching
constant \(\nu_\Omega\) gives a sharp and easily computable upper bound for
these tests.  In all reported cases,
\(\lambda_{\max}^{\rm num}(M_{\Omega{\rm AS}}^{-1}A)\lesssim\nu_\Omega\).  The estimate is nearly
saturated for the \(H^1\) and \(H(\operatorname{div})\) tests, while the
three-dimensional \(H(\operatorname{curl})\) tests show somewhat more slack.
The comparison with coloring constants is also informative.  By construction,
\(\nu_\Omega\le\chi_\Omega\le\chi_\Omega^{\rm MIS}\), but \(\nu_\Omega\) is
computable at scale, unlike \(\chi_\Omega\), since it only requires row sums
of \(\mathcal E_\Omega\).  In the two-dimensional \(H^1\) and
\(H(\operatorname{div})\) tests, \(\nu_\Omega\) and \(\chi_\Omega\) coincide.
In contrast, the three-dimensional \(H(\operatorname{curl})\) refined-mesh
runs give strict separations, e.g.,
\(\nu_\Omega=8<9=\chi_\Omega\) for both \(p=1\) and \(p=2\).  Thus the
row-touching estimate can be genuinely sharper than the corresponding exact
coloring constant, consistent with
\cref{lem:nu_strictly_smaller_than_coloring}.
We note that in certain two-dimensional $H^1$ tests not reported here we have observed a strict separation of $\nu_{\Omega} = 3 < 4 = \chi_{\Omega}$.

\section{Conclusion}
\label{sec:conclusion}

This paper proposes a novel two-level convergence theory for the recently developed LS-AMG-DD solver, for SPD matrices admitting Gram representations, \(A=G^{\top}G\) \cite{southworth2026lsamgdd}.  
The solver constructs interpolation by solving local spectral problems on algebraically-defined nonoverlapping aggregates, and uses block smoothing defined with respect to the aggregate graph.
The central estimate we develop is based on a
weak approximation property in an aggregate-wise block-Jacobi norm. We establish that once the local generalized eigenproblems have been solved, the realized spectral cutoff provides a computable bound for the corresponding approximation constant.  
Combining this estimate with sharp theory for
two-level cycles yields convergence bounds that separate the contribution of the coarse space from the scale of the
smoother.
Numerical experiments across challenging \(H^1\), \(H(\operatorname{div})\), and
\(H(\operatorname{curl})\) finite element problems support the theory.
The numerical tests also indicate that the bounds generally predict the scale of practically relevant constants to within modest factors (e.g., $\sim 1$--3 in many cases), despite being derived from purely algebraic upper bounds.  
The results also show little systematic convergence
deterioration under mesh refinement or increasing polynomial degree refinement,
which is consistent with the theory.
Numerical experiments indicate that nonsymmetric overlapping Schwarz
smoothers can sometimes be more effective than block Jacobi or overlapping
additive Schwarz.  Our general theory does not yet cover these smoothers,
because it depends on the constant
\(\lambda_{\max}(M^{-1}_{\omega{\rm J}}\widetilde M)\), for which we do not currently have a practical estimate.  It is also unclear whether such an estimate would lead to sharp two-level bounds.

\section*{Acknowledgments}
The authors used OpenAI's ChatGPT during the preparation of this manuscript, including for exploratory mathematical discussion, coding, and drafting of the exposition.
All mathematical claims, proofs, computations, citations, and final wording were reviewed, verified, and edited by the authors, who take full responsibility for the manuscript.

\bibliographystyle{siamplain}
\bibliography{conv-thry-refs}

\appendix

\end{document}